\documentclass{article}

\usepackage{amsthm,amsmath,mathtools,bm,enumerate,subcaption}

\usepackage[final,nonatbib]{nips_2018}
\usepackage[numbers]{natbib} 
\usepackage[utf8]{inputenc} % allow utf-8 input
\usepackage[T1]{fontenc}    % use 8-bit T1 fonts
\usepackage{hyperref}       % hyperlinks
\usepackage{url}            % simple URL typesetting
\usepackage{booktabs}       % professional-quality tables
\usepackage{amsfonts}       % blackboard math symbols
\usepackage{nicefrac}       % compact symbols for 1/2, etc.
\usepackage{microtype,color}      % microtypography
\usepackage{graphicx}
\usepackage{animate}

\usepackage[mathcal]{eucal}

\newcommand{\R}{\mathbb{R}}
\newcommand{\N}{\mathbb{N}}

\newcommand{\M}{\mathcal{M}}
\newcommand{\X}{\mathcal{X}}

\newcommand{\V}{\mathcal{V}}

\renewcommand{\P}{\mathcal{P}}
\newcommand{\F}{\mathcal{F}}

\renewcommand{\d}{\mathrm{d}}

\newcommand{\id}{\mathrm{id}}

\newcommand{\qwhereq}{\quad \text{where} \quad}

\renewcommand{\S}{{\mathbb{S}}}

\newcommand{\tint}{{\textstyle{\int}}}

\newcommand{\BL}{\mathrm{BL}}
\newcommand{\Lip}{\mathrm{Lip}}

\DeclareMathOperator{\sign}{sign}

\DeclareMathOperator{\proj}{proj}

\DeclareMathOperator{\spt}{spt}

\newtheorem{theorem}{Theorem}[section]
\newtheorem{proposition}[theorem]{Proposition}
\newtheorem{lemma}[theorem]{Lemma}

\newtheorem{definition}[theorem]{Definition}

\newtheorem{assumptions}[theorem]{Assumptions}
\newtheorem{property}[theorem]{Property}
\graphicspath{ {images/} }

\title{On the Global Convergence of Gradient Descent for Over-parameterized Models using Optimal Transport}

\author{
  L\'ena\"ic Chizat\\
  INRIA, ENS, PSL Research University\\
  Paris, France \\
  \texttt{lenaic.chizat@inria.fr} \\
 \And
  Francis Bach \\
  INRIA, ENS, PSL Research University\\
  Paris, France \\
  \texttt{francis.bach@inria.fr} \\
}

\begin{document}

\maketitle

\begin{abstract}
Many tasks in machine learning and signal processing can be solved by minimizing a convex function of a measure. This includes sparse spikes deconvolution or training a neural network with a single hidden layer. For these problems, we study a simple minimization method: the unknown measure is discretized into a mixture of particles and a continuous-time gradient descent is performed on their weights and positions. This is an idealization of the usual way to train neural networks with a large hidden layer. We show that, when initialized correctly and in the many-particle limit, this gradient flow, although non-convex, converges to global minimizers. The proof involves Wasserstein gradient flows, a by-product of optimal transport theory. Numerical experiments show that this asymptotic behavior is already at play for a reasonable number of particles, even in high dimension.
\end{abstract}

\section{Introduction}\label{sec:introduction}
A classical task in machine learning and signal processing is to search for an element in a Hilbert space $\F$ that minimizes a smooth, convex \emph{loss} function $R:\F\to \R_+$ and that is a linear combination of a few elements from a large given parameterized set $\{\phi(\theta)\}_{\theta \in \Theta}\subset \F$. A general formulation of this problem is to describe the linear combination through an unknown signed measure $\mu$ on the parameter space and to solve for
\begin{equation}\label{eq:introproblem}
J^*=\min_{\mu \in \M(\Theta)} J(\mu), \qquad \text{}\qquad J(\mu) :=R\left( \int \phi \d \mu \right)  + {G(\mu)}
\end{equation}
where $\M(\Theta)$ is the set of signed measures on the parameter space $\Theta$ and $G: \M(\Theta)\to \R$ is an optional convex regularizer, typically the total variation norm when sparse solutions are preferred.  In this paper, we consider the \emph{infinite-dimensional} case where the parameter space $\Theta$ is a domain of $\R^d$ and $\theta \mapsto \phi(\theta)$ is differentiable. This framework covers:
\begin{itemize}
\item 
Training \emph{neural networks with a single hidden layer}, where the goal is to select, within a specific class, a function that maps features in $\R^{d-1}$ to labels in $\R$, from the observation of a joint distribution of features and labels. This corresponds to $\F$ being the space of square-integrable real-valued functions on $\R^{d-1}$, $R$ being, e.g., the quadratic or the logistic loss function, and $\phi(\theta): x\mapsto \sigma (\sum_{i=1}^{d-1} \theta_i x_i + \theta_d)$, with an activation function $\sigma:\R\to\R$. Common choices are the sigmoid function or the rectified linear unit~\cite{haykin1994neural,goodfellow2016deep}, see more details in Section~\ref{subsec:neuralnet}.
\item 
\emph{Sparse spikes deconvolution}, where one attempts to recover a signal which is a mixture of impulses on $\Theta$ given a noisy and filtered observation $y$ (a square-integrable function on $\Theta$). This  corresponds to $\F$ being the space of square-integrable real-valued functions on $\R^{d}$, defining $\phi(\theta): x \mapsto \psi(x-\theta)$ the translations of the filter impulse response $\psi$ and $R(f)= (1/2\lambda)\Vert f-y\Vert^2_{L^2}$, for some $\lambda>0$ that depends on the estimated noise level. Solving~\eqref{eq:introproblem} allows then to reconstruct the mixture of impulses with some guarantees~\cite{de2012exact, duval2015exact}.
\item 
Low-rank tensor decomposition~\cite{haeffele2017global}, recovering mixture models from sketches~\cite{poon2018dual}, see~\cite{boyd2017alternating} for a detailed list of other applications. For example, with symmetric matrices, $\F = \R^{d \times d}$ and $\Phi(\theta) =\theta \theta^\top$,  we recover low-rank matrix decompositions~\cite{srebro}.
\end{itemize}

\subsection{Review of optimization methods and previous work}

While~\eqref{eq:introproblem} is a convex problem, finding approximate minimizers is hard as the variable is infinite-dimensional. Several lines of work provide optimization methods but with strong limitations. 
 \paragraph{Conditional gradient / Frank-Wolfe.} This approach tackles a variant of \eqref{eq:introproblem} where the regularization term is replaced by an upper bound on the total variation norm; the associated constraint set is the convex hull of all Diracs and negatives of Diracs at elements of $\theta \in \Theta$, and thus adapted to conditional gradient algorithms~\cite{jaggi}. At each iteration, one adds a new particle by solving a linear minimization problem over the constraint set (which correspond to finding a particle $\theta \in \Theta$), and then updates the weights. The resulting iterates are sparse and there is a guaranteed sublinear convergence rate of the objective function to its minimum. However, the linear minimization subroutine is hard to perform in general : it is for instance NP-hard for neural networks with homogeneous activations~\cite{bach2017breaking}. One thus generally resorts to space gridding (in low dimension) or to approximate steps, akin to boosting~\cite{wang2015functional}.  The practical behavior is improved with nonconvex updates~\cite{boyd2017alternating,bredies2013inverse} reminiscent of the flow studied below.
\paragraph{Semidefinite hierarchy.}  Another approach is to parameterize the unknown measure by its sequence of moments. The space of such sequences is characterized by a hierarchy of SDP-representable necessary conditions. This approach concerns a large class of \emph{generalized moment problems}~\cite{lasserre2010moments} and can be adapted to deal with special instances of \eqref{eq:introproblem}~\cite{catala2017low}. It is however restricted to $\phi$ which are combinations of few polynomial moments, and its complexity explodes exponentially with the dimension $d$. For $d\geq 2$, convergence to a global minimizer is only guaranteed asymptotically, similarly to the results of the present paper.
\paragraph{Particle gradient descent.} A third approach, which exploits the differentiability of $\phi$, consists in discretizing the unknown measure $\mu$ as a mixture of $m$ particles parameterized by their positions and weights. This corresponds to the finite-dimensional problem
\begin{equation}\label{eq:introdiscrete}
\min_{\substack{\bm{w}\in \R^m \\ \bm{\theta}\in \Theta^m}} J_m (\bm{w},\bm{\theta})
 \qquad \text{where} \qquad
 J_m(\bm{w},\bm{\theta}) := J\left(\frac1m \sum_{i=1}^m w_i \delta_{\theta_i}\right),
\end{equation}
which can then be solved by classical gradient descent-based algorithms.
This method is simple to implement and is widely used for the task of neural network training but, a priori, we may only hope to converge to local minima since $J_m$ is non-convex. \emph{Our goal is to show that this method also benefits from the convex structure of~\eqref{eq:introproblem} and enjoys an asymptotical global optimality guarantee}.

There is a recent literature on global optimality results for~\eqref{eq:introdiscrete} in the specific task of training neural networks. 
It is known that in this context, $J_m$ has less, or no, local minima in an over-parameterization regime and stochastic gradient descent (SGD) finds a global minimizer under restrictive assumptions~\cite{soudry2017exponentially, venturi2018neural, soltanolkotabi2017theoretical,li2017convergence}; see~\cite{soltanolkotabi2017theoretical} for an account of recent results. 
Our approach is not directly comparable to these works: it is more abstract and nonquantitative---we study an ideal dynamics that one can only hope to approximate---but also much more generic. Our objective, in the space of measures, \emph{has} many local minima, but we build gradient flows that avoids them, relying mainly on the homogeneity properties of $J_m$ (see ~\cite{haeffele2017global,journee2010low} for other uses of homogeneity in non-convex optimization). The novelty is to see~\eqref{eq:introdiscrete} as a discretization of~\eqref{eq:introproblem}---a point of view also present in~\cite{nitanda2017stochastic} but not yet exploited for global optimality guarantees.
\subsection{Organization of the paper and summary of contributions}
Our goal is to explain when and why the non-convex particle gradient descent finds global minima. We do so by studying the many-particle limit $m\to \infty$ of the gradient flow of $J_m$. More specifically:
\begin{itemize}
\item 
In Section~\ref{sec:manyparticle}, we introduce a more general class of problems and study the many-particle limit of the associated particle gradient flow. This limit is characterized as a \emph{Wasserstein gradient flow} (Theorem~\ref{th:manyparticle}), an object which is a by-product of optimal transport theory.
\item 
In Section~\ref{sec:convergence}, under assumptions on $\phi$ and the initialization, we prove that if this Wasserstein gradient flow converges, then the limit is a global minimizer of $J$. Under the same conditions, it follows that if $(\bm{w}^{(m)}(t),\bm{\theta}^{(m)}(t))_{t\geq 0}$ are gradient flows for $J_m$ suitably initialized, then
\[
\lim_{m,t\to \infty} J(\mu_{m,t}) = J^* \qquad \text{where} \qquad \mu_{m,t} = \frac1m\sum_{i=1}^m w^{(m)}_i(t) \delta_{\theta^{(m)}_i(t)}.
\]
\item 
Two different settings that leverage the structure of $\phi$ are treated: the \emph{$2$-homogeneous} and the \emph{partially $1$-homogeneous} case. In Section~\ref{sec:case}, we apply these results to sparse deconvolution and training neural networks with a single hidden layer, with sigmoid or ReLU activation function. In each case, our result prescribes conditions on the initialization pattern.
\item 
We perform simple numerical experiments that indicate that this asymptotic regime is already at play for small values of $m$, even for high-dimensional problems. The method behaves incomparably better than simply optimizing on the weights with a very large set of fixed particles.
\end{itemize}
Our focus on qualitative results might be surprising for an optimization paper, but we believe that this is an insightful first step given the hardness and the generality of the problem. We suggest to understand our result as a first \emph{consistency principle} for practical and a commonly used non-convex optimization methods. While we focus on the idealistic setting of a \emph{continuous-time} gradient flow with \emph{exact} gradients, this is expected to reflect the behavior of first order descent algorithms, as they are known to approximate the former: see \cite{scieur2017integration} for (accelerated) gradient descent and~\cite[Thm. 2.1]{kushner2003stochastic} for SGD.
\paragraph{Notation.}
Scalar products and norms are denoted by $\cdot$ and $\vert \cdot \vert$ respectively in $\R^d$, and by $\langle\cdot,\cdot\rangle$ and $\Vert\cdot \Vert$ in the Hilbert space $\F$. Norms of linear operators are also denoted by $\Vert \cdot \Vert$. The differential of a function $f$ at a point $x$ is denoted $df_x$.  We write $\M(\R^d)$ for the set of finite signed Borel measures on~$\R^d$,  $\delta_x$ is a Dirac mass at a point $x$ and $\P_2(\R^d)$ is the set of probability measures endowed with the Wasserstein distance $W_2$ (see Appendix~\ref{app:intro}).
\paragraph{Recent related work.} Several independent works~\cite{mei2018mean, rotskoff2018neural, sirignano2018mean} have studied the many-particle limit of training a neural network with a single large hidden layer and a quadratic loss $R$. Their main focus is on quantifying the convergence of SGD or noisy SGD to the limit trajectory, which is precisely a mean-field limit in this case. Since in our approach this limit is mostly an intermediate step necessary to state our global convergence theorems, it is not studied extensively for itself. These papers thus provide a solid complement to Section~\ref{subsec:manyparticle} (a difference is that we do not assume that $R$ is quadratic nor that $V$ is differentiable). Also,~\cite{mei2018mean} proves a quantitive global convergence result for noisy SGD to an approximate minimizer: we stress that our results are of a different nature, as they rely on homogeneity and not on the mixing effect of noise.
\section{Particle gradient flows and many-particle limit}\label{sec:manyparticle}
\subsection{Main problem and assumptions}\label{subsec:lifting}
From now on, we consider the following class of problems on the space of \emph{non-negative} finite measures on a domain $\Omega \subset \R^d$ which, as explained below, is more general than~\eqref{eq:introproblem}:
\begin{equation}\label{eq:mainproblem}
F^* =  \min_{\mu \in \M_+(\Omega)} F(\mu) \qquad \text{where} \qquad F(\mu) = R\left(\int \Phi \d \mu\right) + \int V \d \mu,
\end{equation}
and we make the following assumptions.
\begin{assumptions}\label{ass:regularity} $\F$ is a separable Hilbert space, $\Omega \subset \R^{d}$ is the closure of a convex open set, and
\begin{enumerate}[(i)]
\item \emph{(smooth loss)} 
$R:\F\to \R_+$ is differentiable, with a differential $dR$ that is Lipschitz on bounded sets and bounded on sublevel sets,
\item \emph{(basic regularity)}
$\Phi:\Omega\to \F$ is (Fr\'echet) differentiable, $V:\Omega \to \R_+$ is semiconvex\footnote{A function $f: \R^d \to \R$ is semiconvex, or $\lambda$-convex, if $f+ \lambda \vert \cdot \vert^2$ is convex, for some $\lambda\in \R$. On a compact domain, any smooth fonction is semiconvex.}, and 
\item \emph{(locally Lipschitz derivatives with sublinear growth)}
there exists a family $(Q_r)_{r>0}$ of nested nonempty closed convex subsets of $\Omega$ such that:
\begin{enumerate}
\item 
$\{u\in \Omega \;;\; \mathrm{dist}(u,Q_r)\leq r'\} \subset Q_{r+r'}$ for all $r,r'>0$,
\item
$\Phi$ and $V$ are bounded and $d\Phi$ is Lipschitz on each $Q_r$, and
\item 
there exists $C_1,C_2>0$ such that $\sup_{u \in Q_r}( \Vert d\Phi_u\Vert + \Vert \partial V(u)\Vert) \leq C_1 + C_2r$ for all $r>0$, where $\Vert \partial V(u)\Vert$ stands for the maximal norm of an element in $\partial V(u)$.
\end{enumerate}\label{subass:Qr}
\end{enumerate}
\end{assumptions}
Assumption \ref{ass:regularity}-\eqref{subass:Qr} reduces to classical local Lipschitzness and growth assumptions on $d\Phi$ and $\partial V$ if the nested sets $(Q_r)_r$ are the balls of radius $r$, but unbounded sets $Q_r$ are also allowed. These sets are a technical tool used later to confine the gradient flows in areas where gradients are well-controlled. By convention, we set $F(\mu)=\infty$ if $\mu$ is not concentrated on $\Omega$. Also, the integral $\int \Phi \d \mu$ is a \emph{Bochner integral}~\cite[App. E6]{cohn1980measure}. It yields a well-defined value in $\F$ whenever $\Phi$ is measurable and $\int \Vert \phi\Vert \d \vert \mu\vert <\infty$. Otherwise, we also set $F(\mu)=\infty$ by convention.

\paragraph{Recovering~\eqref{eq:introproblem} through lifting.}
It is shown in Appendix~\ref{app:lifting} that, for a class of admissible regularizers $G$ containing the total variation norm, problem~\eqref{eq:introproblem} admits an equivalent formulation as~\eqref{eq:mainproblem}. Indeed, consider the \emph{lifted} domain $\Omega = \R\times \Theta$, the function $\Phi(w,\theta)=w\phi(\theta)$ and $V(w,\theta)=\vert w\vert$. Then $J^*$ equals $F^*$ and given a minimizer of one of the problems, one can easily build minimizers for the other. This equivalent \emph{lifted} formulation removes the asymmetry between weight and position---weight becomes just another coordinate of a particle's position. This is the right point of view for our purpose and this is why $F$ is our central object of study in the following.

\paragraph{Homogeneity.} The functions $\Phi$ and $V$ obtained through the lifting share the property of being positively $1$-homogeneous in the variable $w$. A function $f$ between vector spaces is said positively $p$-homogeneous when for all $\lambda>0$ and argument $x$, it holds $f(\lambda x)=\lambda^pf(x)$. This property is central for our global convergence results (but is not needed throughout Section~\ref{sec:manyparticle}).

\subsection{Particle gradient flow}
We first consider an initial measure which is a mixture of particles---an atomic measure--- and define the initial object in our construction: the \emph{particle gradient flow}.
For a number $m\in \N$ of particles, and a vector $\mathbf u\in \Omega^m$ of positions, this is the gradient flow of
\begin{equation}\label{eq:finiteparticle}
F_m(\mathbf{u}) \coloneqq F\left(\frac1m \sum_{i=1}^m \delta_{\mathbf u_i}\right) = R \left( \frac{1}{m} \sum_{i=1}^m \Phi(\mathbf u_i) \right) + \frac1m \sum_{i=1}^m V(\mathbf  u_i),
\end{equation}
or, more precisely, its \emph{subgradient flow} because $V$ can be non-smooth. We recall that a \emph{subgradient} of a (possibly non-convex) function $f: \R^d \to \bar \R$ at a point $u_0\in \R^d$ is a $p\in \R^d$ satisfying $f(u)\geq f(u_0) + p\cdot (u-u_0) + o(u-u_0)$ for all $u \in \R^d$. The set of subgradients at $u$ is a closed convex set called the \emph{subdifferential} of $f$ at $u$ denoted $\partial f(u)$~\cite{rockafellar97}.

\begin{definition}[Particle gradient flow]\label{def:classicalGF}
A gradient flow for the functional $F_m$ is an absolutely continuous\footnote{An absolutely continuous function $x:\R\to \R^d$ is almost everywhere differentiable and satisfies $x(t)-x(s) = \int_s^t x'(r)dr$ for all $s<t$.} path $\mathbf u : \R_+\to \Omega^m$ which satisfies $\mathbf u'(t) \in - m \,\partial F_m(\mathbf u(t))$ for almost every $t\geq0$.
\end{definition}

This definition uses a subgradient scaled by $m$, which is the subgradient relative to the scalar product on $(\R^d)^m$ scaled by $1/m$: this normalization amounts to assigning a mass $1/m$ to each particle and is convenient for taking the many-particle limit $m\to \infty$. 
We now state basic properties of this object.
\begin{proposition}\label{eq:classicalGF}
For any initialization $\mathbf u(0)\in \Omega^m$, there exists a unique gradient flow $\mathbf u : \R_+\to \Omega^m$ for $F_m$. Moreover, for almost every $t>0$, it holds $\frac{d}{ds} F_m(\mathbf{u}(s))\vert_{s=t}=-\vert \mathbf{u}'(t)\vert^2$ and the velocity of the $i$-th particle is given by $\mathbf{u}_i'(t) = v_t(\mathbf{u}_i(t))$, where for $u\in \Omega$ and $\mu_{m,t}:=(1/m)\sum_{i=1}^m \delta_{\mathbf{u}_i(t)}$,
\begin{equation}\label{eq:velocity}
v_t(u) = \tilde v_t(u) -\proj_{\partial V(u)}(\tilde v_t(u)) \quad \text{with}\quad \tilde v_t(u) = - \left[\left\langle R'\left(\tint \Phi \d \mu_{m,t}\right), \partial_j \Phi(u) \right\rangle \right]_{j=1}^d.
\end{equation}
\end{proposition}
The expression of the velocity involves a projection because gradient flows select subgradients of minimal norm~\cite{santambrogio2015optimal}. We have denoted by $R'(f)\in \F$ the gradient of $R$ at $f\in \F$ and by $\partial_j \Phi(u)\in \F$ the differential $d\Phi_u$ applied to the $j$-th vector of the canonical basis of $\R^d$.  Note that $[\tilde v_t(\mathbf{u}_i)]_{i=1}^m$ is (minus) the gradient of the first term in~\eqref{eq:finiteparticle} : when $V$ is differentiable, we have $v_t(u)=\tilde v_t(u) -\nabla V(u)$ and we recover the classical gradient of~\eqref{eq:finiteparticle}. When $V$ is non-smooth, this gradient flow can be understood as a continuous-time version of the forward-backward minimization algorithm~\cite{combettes2011proximal}.

\subsection{Wasserstein gradient flow}

The fact that the velocity of each particle can be expressed as the evaluation of a velocity field (Eq.~\eqref{eq:velocity}) makes it easy, at least formally, to generalize the particle gradient flow to arbitrary measure-valued initializations---not just atomic ones.  On the one hand, the evolution of a time-dependent measure $(\mu_t)_t$ under the action of instantaneous velocity fields $(v_t)_{t\geq 0}$ can be formalized by a conservation of mass equation, known as the \emph{continuity equation}, that reads $\partial_t \mu_t = - \mathrm{div} (v_t\mu_t)$ where $\mathrm{div}$ is the divergence operator\footnote{For a smooth vector field $E=(E_i)_{i=1}^d:\R^d \to \R^d$, its divergence is  given by $\mathrm{div}(E) = \sum_{i=1}^d \partial E_i/\partial x_i$.} (see Appendix~\ref{app:WGF}).
On the other hand, there is a direct link between the velocity field~\eqref{eq:velocity} and the functional $F$. The differential of $F$ evaluated at $\mu\in \M(\Omega)$ is represented by the function $F'(\mu): \Omega \to \R$ defined as
\[
F'(\mu)(u) := \left\langle R'\left(\int \Phi\d\mu\right), \Phi(u)\right\rangle + V(u).
\]
Thus $v_t$ is simply a field of (minus) subgradients of $F'(\mu_{m,t})$---it is in fact the field of minimal norm subgradients. We write this relation $v_t \in - \partial F'(\mu_{m,t})$. The set $\partial F'$ is called the \emph{Wasserstein subdifferential} of $F$, as it can be interpreted as the subdifferential of $F$ relatively to the Wasserstein metric on $\P_2(\Omega)$ (see Appendix~\ref{app:WGFproperties}).
We thus expect that for initializations with arbitrary probability distributions, the generalization of the gradient flow coindices with the following object.

\begin{definition}[Wasserstein gradient flow]\label{def:WassersteinGF}
A Wasserstein gradient flow for the functional $F$ on a time interval ${[0,T[}$ is an absolutely continuous path $(\mu_t)_{t\in {[0,T[}}$ in $\P_2(\Omega)$ that satisfies, distributionally on ${[0,T[}\times \Omega^d$,
\begin{equation}\label{eq:gradientflow}
\partial_t \mu_t = - \mathrm{div} (v_t\mu_t) \qwhereq v_t \in - \partial F'(\mu_t) .
\end{equation}
\end{definition}

This is a proper generalization of Definition~\ref{def:classicalGF} since, whenever $(\mathbf u(t))_{t\geq 0}$ is a particle gradient flow for $F_m$, then $t\mapsto \mu_{m,t} := \frac1m \sum_{i=1}^m \delta_{\mathbf u_i{(t)}}$ is a Wasserstein gradient flow for $F$ in the sense of Definition~\ref{def:WassersteinGF} (see Proposition~\ref{prop:atomicWGF}). 
By leveraging the abstract theory of gradient flows developed in~\cite{ambrosio2008gradient}, we show in Appendix~\ref{app:WGFproperties} that these Wasserstein gradient flows are well-defined.

\begin{proposition}[Existence and uniqueness]\label{prop:uniquenessWGF}
Under Assumptions~\ref{ass:regularity}, if $\mu_0\in \P_2(\Omega)$ is concentrated on a set $Q_{r_0}\subset \Omega$, then there exists a unique Wasserstein gradient flow $(\mu_t)_{t\geq 0}$ for $F$ starting from~$\mu_0$. It satisfies the continuity equation with the velocity field defined in~\eqref{eq:velocity} (with $\mu_t$ in place of $\mu_{m,t}$).
\end{proposition}
Note that the condition on the initialization is automatically satisfied in Proposition~\ref{eq:classicalGF} because there the initial measure has a finite discrete support: it is thus contained in any $Q_r$ for $r>0$ large enough. 

\subsection{Many-particle limit}\label{subsec:manyparticle}
We now characterize the many-particle limit of classical gradient flows, under Assumptions~\ref{ass:regularity}.

\begin{theorem}[Many-particle limit]\label{th:manyparticle}
Consider $(t\mapsto \mathbf u_{m}(t))_{m\in \mathbb{N}}$ a sequence of classical gradient flows for $F_m$ initialized in a set $Q_{r_0}\subset \Omega$. If $\mu_{m,0}$ converges to some $\mu_0\in\P_2(\Omega)$ for the Wasserstein distance $W_2$, then $(\mu_{m,t})_t$ converges, as $m\to \infty$, to the unique Wasserstein gradient flow of $F$ starting from $\mu_0$.
\end{theorem}

Given a measure $\mu_0\in \P_2(Q_{r_0})$, an example for the sequence $\mathbf{u}_{m}(0)$ is $\mathbf {u}_{m}(0) = (u_1,\dots, u_m)$ where $u_1,u_2,\dots, u_m$ are independent samples distributed according to $\mu_0$. By the law of large numbers for empirical distributions, the sequence of empirical distributions $\mu_{m,0}=\frac1m\sum_{i=1}^m \delta_{u_i}$ converges (almost surely, for $W_2$) to $\mu_0$. In particular, our proof of Theorem~\ref{th:manyparticle} gives an alternative proof of the existence claim in Proposition~\ref{prop:uniquenessWGF} (the latter remains necessary for the uniqueness of the limit).

%%%%%%%%%%%%%%%%%%%%%%%%%%%%%%%%%%%%%%%%%%%%%%%%%%%%%%%%%%%
\section{Convergence to global minimizers}\label{sec:convergence}
\subsection{General idea}
As can be seen from Definition \ref{def:WassersteinGF}, a probability measure $\mu \in \P_2(\Omega)$ is a stationary point of a Wasserstein gradient flow if and only if $ 0 \in \partial F'(\mu)(u)  \mbox{ for } \mu\mbox{-a.e. } u\in \Omega$. It is proved in~\cite{nitanda2017stochastic} that these stationary points are, in some cases, optimal over probabilities that have a smaller support. However, they are not in general global minimizers of $F$ over $\M_+(\Omega)$, even when $R$ is convex. Such global minimizers are indeed characterized as follows.
\begin{proposition}[Minimizers]\label{prop:stationnary}
Assume that $R$ is convex. A measure $\mu\in \M_+(\Omega)$ such that $F(\mu)<\infty$  minimizes $F$ on $\M_+(\Omega)$ iff $F'(\mu)\geq 0$ and $F'(\mu)(u)=0$ for $\mu$-a.e.\ $u\in \Omega$.
\end{proposition}
Despite these strong differences between stationarity and global optimality, we show in this section that Wasserstein gradient flows converge to global minimizers, under two main conditions:
\begin{itemize}
\item \emph{On the structure}:  $\Phi$ and $V$ must share a homogeneity direction (see Section~\ref{subsec:lifting} for the definition of homogeneity), and 
\item \emph{On the initialization}: the support of the initialization of the Wasserstein gradient flow satisfies a ``separation'' property. This property is preserved throughout the dynamic and, combined with homogeneity, allows to escape from neighborhoods of non-optimal points.
\end{itemize}
We turn these general ideas into concrete statements for two cases of interest, that exhibit different structures and behaviors: (i) when $\Phi$ and $V$ are positively $2$-homogeneous and (ii) when $\Phi$ and $V$ are positively $1$-homogeneous with respect to one variable.

\subsection{The $2$-homogeneous case}
In the $2$-homogeneous case a rich structure emerges, where the $(d-1)$-dimensional sphere $\S^{d-1} \subset \R^d$ plays a special role. This covers the case of lifted problems of Section~\ref{subsec:lifting} when $\phi$ is $1$-homogeneous and neural networks with ReLU activation functions.

\begin{assumptions}\label{ass:2homogeneous}
The domain is $\Omega=\R^{d}$ with $d\geq 2$ and $\Phi$ is differentiable with $d\Phi$ locally Lipschitz, $V$ is semiconvex and $V$ and $\Phi$ are both positively $2$-homogeneous. Moreover, 
\begin{enumerate}[(i)]
\item \emph{(smooth convex loss)} The loss $R$ is convex, differentiable with differential $dR$ Lipschitz on bounded sets and bounded on sublevel sets,\label{subass:loss}
\item \emph{(Sard-type regularity)} For all $f\in \F$, the set of regular values\footnote{For a function $g:\Theta\to\R$, a regular value is a real number $\alpha$ in the range of $g$ such that $g^{-1}(\alpha)$ is included in an open set where $g$ is differentiable and where $dg$ does not vanish.} of $\theta \in \S^{d-1}\mapsto \langle f, \Phi(\theta)\rangle + V(\theta)$ is dense in its range (it is in fact sufficient that this holds for functions $f$ which are of the form $f = R'(\int \Phi \d \mu)$ for some $\mu \in \M_+(\Omega)$). \label{subass:sard}
\end{enumerate}
\end{assumptions}
Taking the balls of radius $r>0$ as the family $(Q_r)_{r>0}$, these assumptions imply Assumptions~\ref{ass:regularity}. We believe that Assumption~\ref{ass:2homogeneous}-\eqref{subass:sard} is not of practical importance: it is only used to avoid some pathological cases in the proof of Theorem~\ref{th:mainhomogeneous}. By applying Morse-Sard's lemma~\cite{abraham1967transversal}, it is anyways fulfilled if the function in question is $d-1$ times continuously differentiable. 
We now state our first global convergence result. It involves a condition on the initialization, a \emph{separation} property, that can only be satisfied in the many-particle limit. In an ambient space $\Omega$, we say that a set $C$ \emph{separates} the sets $A$ and $B$ if any continuous path in $\Omega$ with endpoints in $A$ and $B$ intersects $C$.

\begin{theorem}\label{th:mainhomogeneous}
Under Assumptions~\ref{ass:2homogeneous}, let $(\mu_t)_{t\geq 0}$ be a Wasserstein gradient flow of $F$ such that, for some $0<r_a < r_b$, the support of $\mu_0$ is contained in $B(0,r_b)$ and separates the spheres $r_a \S^{d-1}$ and $r_b \S^{d-1}$. If $(\mu_t)_t$ converges to $\mu_\infty$ in $W_2$, then $\mu_\infty$ is a global minimizer of $F$ over $\M_+(\Omega)$. In particular, if $(\mathbf{u}_m(t))_{m\in \N,t\geq 0}$ is a sequence of classical gradient flows initialized in $B(0,r_b)$ such that $\mu_{m,0}$ converges weakly to $\mu_0$ then (limits can be interchanged)
\[
\lim_{t,m\to \infty} F(\mu_{m,t}) = \min_{\mu \in \M_+(\Omega)} F(\mu).
\]
\end{theorem}

A proof and stronger statements are presented in Appendix~\ref{app:global}. There, we give a criterion for Wasserstein gradient flows to escape neighborhoods of non-optimal measures---also valid in the finite-particle setting---and then show that it is always satisfied by the flow defined above. We also weaken the assumption that $\mu_t$ converges: we only need a certain projection of $\mu_t$ to converge weakly. Finally, the fact that limits in $m$ and $t$ can be interchanged is not anecdotal: it shows that the convergence is not conditioned on a relative speed of growth of both parameters.

This result might be easier to understand by drawing an informal distinction between (i) the structural assumptions which are instrumental and (ii) the technical conditions which have a limited practical interest. The initialization and the homogeneity assumptions are of the first kind. The Sard-type regularity is in contrast a purely technical condition: it is generally hard to check and known counter-examples involve artificial constructions such as the Cantor function~\cite{whitney1935function}. Similarly, when there is compactness, a gradient flow that does not converge is an unexpected (in some sense adversarial) behavior, see a counter-example in~\cite{absil2005convergence}. We were however not able to exclude this possibility under interesting assumptions (see a discussion in Appendix \ref{app:GFconvergence}).

\subsection{The partially $1$-homogeneous case}
Similar results hold in the partially $1$-homogeneous setting, which covers the lifted problems of Section~\ref{subsec:lifting} when $\phi$ is bounded (e.g., sparse deconvolution and neural networks with sigmoid activation).
\begin{assumptions}\label{ass:bounded} The domain is $\Omega = \R\times \Theta$ with $\Theta \subset \R^{d-1}$, $\Phi(w,\theta) = w \cdot \phi(\theta)$ and $V(w,\theta)=\vert w \vert \tilde V(\theta)$ where $\phi$ and $\tilde V$ are bounded, differentiable with Lipschitz differential. Moreover, 
\begin{enumerate}[(i)]
\item \emph{(smooth convex loss)} The loss $R$ is convex, differentiable with differential $dR$ Lipschitz on bounded sets and bounded on sublevel sets,
\item  \emph{(Sard-type regularity)}
For all $f\in \F$, the set of regular values of $g_f: \theta \in \Theta \mapsto \langle f, \phi(\theta) \rangle + \tilde V(\theta)$ is dense in its range, and
\item \emph{(boundary conditions)} The function $\phi$ behaves nicely at the boundary of the domain: either 
\begin{enumerate}
\item $\Theta=\R^{d-1}$ and for all $f\in \F$, $\theta \in \S^{d-2}\mapsto g_f(r\theta)$ converges, uniformly in $C^1(\S^{d-2})$ as $r\to \infty$, to a function satisfying the Sard-type regularity,  
or \label{subass:boundaryunbounded}
\item $\Theta$ is the closure of an bounded open convex set and for all $f\in \F$, $g_f$ satisfies Neumann boundary conditions (i.e., for all $\theta \in \partial\Theta$, $d (g_f)_\theta(\vec n_\theta)=0$ where $\vec n_\theta \in \R^{d-1}$ is the normal to $\partial \Theta$ at $\theta$).
\end{enumerate}
\label{subass:boundary}
\end{enumerate}
\end{assumptions}
With the family of nested sets $Q_r:=[-r,r]\times \Theta$, $r>0$, these assumptions imply Assumptions~\ref{ass:regularity}.
The following theorem mirrors the statement of Theorem~\ref{th:mainhomogeneous}, but with a different condition on the initialization. The remarks after Theorem~\ref{th:mainhomogeneous} also apply here.
\begin{theorem}\label{th:mainbounded}
Under Assumptions~\ref{ass:bounded}, let $(\mu_t)_{t\geq 0}$ be a Wasserstein gradient flow of $F$ such that for some $r_0>0$, the support of $\mu_0$ is contained in $[-r_0,r_0] \times \Theta$ and separates $\{-r_0\}\times \Theta$ from $\{r_0\}\times \Theta$.
 If $(\mu_t)_t$ converges to $\mu_\infty$ in $W_2$, then $\mu_\infty$ is a global minimizer of $F$ over $\M_+(\Omega)$. In particular, if $(\mathbf{u}_m(t))_{m\in \N,t\geq 0}$ is a sequence of classical gradient flows initialized in $[-r_0,r_0]\times \Theta$ such that $\mu_{m,0}$ converges to $\mu_0$ in $W_2$ then (limits can be interchanged)
\[
\lim_{t,m\to \infty} F(\mu_{m,t}) = \min_{\mu \in \M_+(\Omega)} F(\mu).
\]
\end{theorem}

%%%%%%%%%%%%%%%%%%%%%%%%
\section{Case studies and numerical illustrations}\label{sec:case}
In this section, we apply the previous abstract statements to specific examples and show on synthetic experiments that the particle-complexity to reach global optimality is very favorable.

\subsection{Sparse deconvolution}\label{subsec:spikes}
For sparse deconvolution, it is typical to consider a signal $y \in \F := L^2(\Theta)$ on the $d$-torus $\Theta=\R^d/\mathbb{Z}^d$. The loss function is $R(f)= (1/2\lambda)\Vert y - f\Vert^2_{L^2}$ for some $\lambda>0$, a parameter that increases with the noise level and the regularization is $V(w,\theta)=\vert w\vert$. Consider a filter impulse response $\psi: \Theta \to \R$ and let $ \Phi(w,\theta): x \mapsto w \cdot \psi(x-\theta)$. The object sought after is a signed measure on $\Theta$, which is obtained from a probability measure on $\R\times \Theta$ by applying a operator defined by $h_1(\mu)(B)= \int_\R w \d \mu(w,B)$ for all measurable $B \subset \Theta$. We show in Appendix~\ref{app:case} that Theorem~\ref{th:mainbounded} applies.

\begin{proposition}[Sparse deconvolution]
Assume that the filter impulse response $\psi$ is $\min\{ 2, d\}$ times continuously differentiable, and that the support of $\mu_0$ contains $\{0\}\times \Theta$. If the projection $(h^1(\mu_t))_t$ of the Wasserstein gradient flow of $F$ weakly converges to $\nu \in \M(\Theta)$, then $\nu$ is a global minimizer of
\[
\min_{\mu \in \M(\Theta)} \frac{1}{2\lambda} \left\Vert y - \tint \psi \d \mu \right\Vert^2_{L^2} + \vert \mu\vert(\Theta).
\]
\end{proposition}

We show an example of such a reconstruction on the $1$-torus on Figure~\ref{fig:spikes}, where the ground truth consists of $m_0=5$ weighted spikes, $\psi$ is an ideal low pass filter (a Dirichlet kernel of order $7$) and $y$ is a noisy observation of the filtered spikes. The particle gradient flow is integrated with the forward-backward algorithm~\cite{combettes2011proximal} and the particles initialized on a uniform grid on $\{0\}\times \Theta$. 

\begin{figure}[h]
\centering
\begin{subfigure}{0.32\textwidth}
\centering
\includegraphics[scale=0.46,trim=0.8cm 0.2cm 0.8cm 0.2cm,clip]{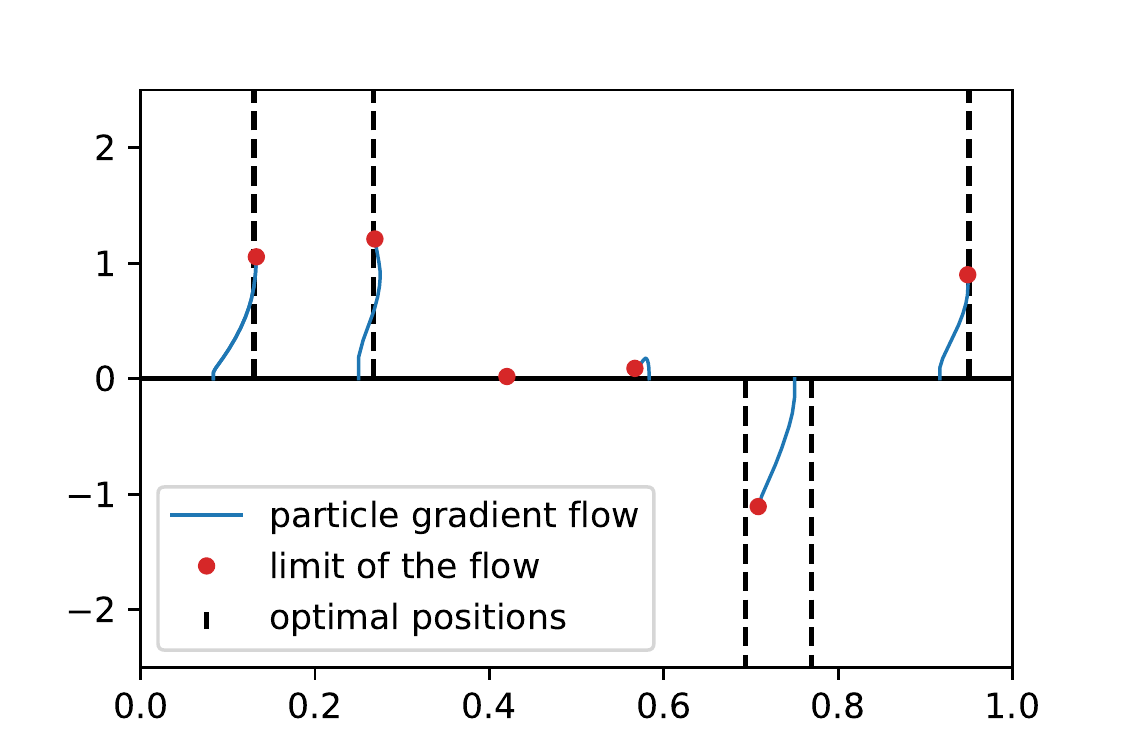}
\end{subfigure}
\begin{subfigure}{0.32\textwidth}
\centering
\includegraphics[scale=0.46,trim=0.8cm 0.2cm 0.8cm 0.2cm,clip]{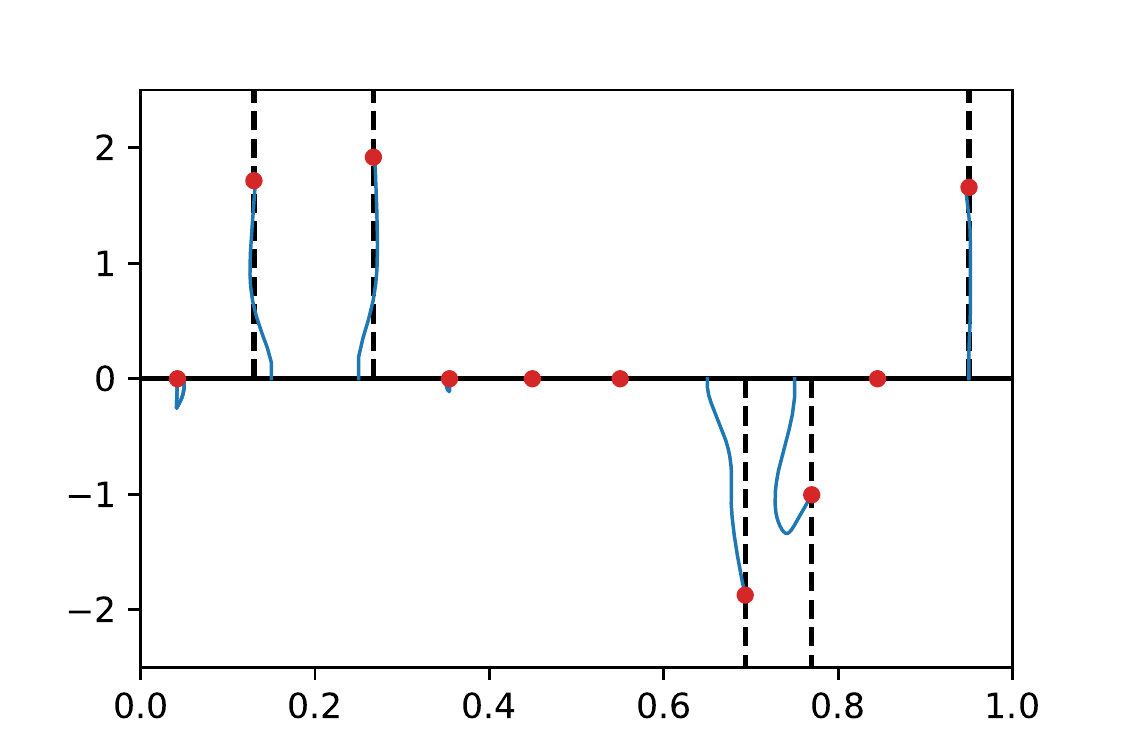}
\end{subfigure}
\begin{subfigure}{0.32\textwidth}
\centering
\includegraphics[scale=0.46,trim=0.8cm 0.2cm 0.8cm 0.2cm,clip]{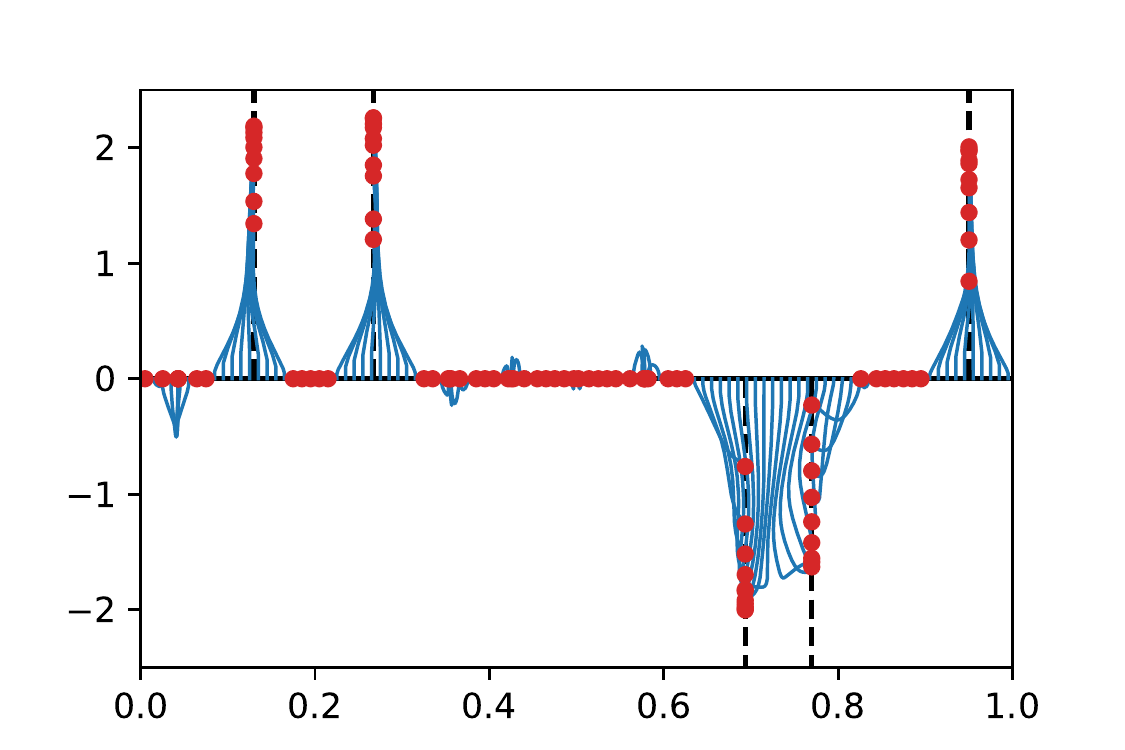}
\end{subfigure}
\caption{Particle gradient flow for sparse deconvolution on the $1$-torus (horizontal axis shows positions, vertical axis shows weights). Failure to find a minimizer with $6$ particles, success with $10$ and $100$ particles (an animated plot of this particle gradient flow can be found in Appendix~\ref{app:details}).}
\label{fig:spikes}
\end{figure}

\subsection{Neural networks with a single hidden layer}\label{subsec:neuralnet}
We consider a joint distribution of features and labels $\rho \in \P(\R^{d-2}\times \R)$ and $\rho_x \in \P(\R^{d-2})$ the marginal distribution of features. The loss is the expected risk $R(f)= \int \ell(f(x),y)\d \rho(x,y)$ defined on $\F = L^2(\rho_x)$, where $\ell: \R\times \R \to \R_+$ is either the squared loss or the logistic loss. Also, we set $\Phi(w,\theta): x \mapsto w \sigma(\sum_{i=1}^{d-2} \theta_i x_i + \theta_{d-1})$ for an activation function $\sigma: \R\to \R$. Depending on the choice of $\sigma$, we face two different situations.

\paragraph{Sigmoid activation.}
If $\sigma$ is a sigmoid, say $\sigma(s) = (1+e^{-s})^{-1}$, then Theorem~\ref{th:mainbounded}, with domain $\Theta = \R^{d-1}$ applies. The natural (optional) regularization term is $V(w,\theta)= \vert w \vert$, which amounts to penalizing the $\ell^1$ norm of the weights.
\begin{proposition}[Sigmoid activation]\label{prop:sigmoid}
Assume that $\rho_x$ has finite moments up to order $\min\{ 4, 2d-2\}$, that the support of $\mu_0$ is $\{0\}\times \Theta$ and that boundary condition~\ref{ass:bounded}-(iii)-(a) holds. If the Wasserstein gradient flow of $F$ converges in $W_2$ to $\mu_\infty$, then $\mu_\infty$ is a global minimizer of $F$.
\end{proposition}
Note that we have to explicitly assume the boundary condition~\ref{ass:bounded}-(iii)-(a) because the Sard-type regularity at infinity cannot be checked \emph{a priori} (this technical detail is discussed in Appendix~\ref{appsubsec:sigmoid}).

\paragraph{ReLU activation.} The activation function $\sigma(s)= \max\{0, s\}$ is positively $1$-homogeneous: this makes $\Phi$ $2$-homogeneous and corresponds, at a formal level, to the setting of Theorem~\ref{th:mainhomogeneous}. An admissible choice of regularizer here would be the (semi-convex) function $V(w,\theta) = \vert w\vert \cdot \vert \theta \vert$~\cite{bach2017breaking}. 
However, as shown in Appendix~\ref{app:ReLU}, the differential $d\Phi$ has discontinuities: this prevents altogether from defining gradient flows, even in the finite-particle regime. 

Still, a statement holds for a different parameterization of the \emph{same class} of functions, which makes $\Phi$ differentiable. To see this, consider a domain $\Theta$ which is the disjoint union of $2$ copies of $\R^{d}$. On the first copy, define $\Phi(\theta): x \mapsto  \sigma(\sum_{i=1}^{d-1} s(\theta_i) x_i + s(\theta_{d}))$ where $s(\theta_i)=  \theta_i \vert \theta_i\vert$ is the signed square function. On the second copy, $\Phi$ has the same definition but with a minus sign. This trick allows to have the same expression power than classical ReLU networks. In practice, it corresponds to simply putting, say, random signs in front of the activation. The regularizer here can be $V(\theta)=\vert \theta \vert^2$.

\begin{proposition}[Relu activation]\label{prop:relu}
Assume that $\rho_x \in \P(\R^{d-1})$ has finite second moments, that the support of $\mu_0$ is $r_0\S^{d-1}$ for some $r_0>0$ (on both copies of $\R^{d}$) and that the Sard-type regularity Assumption~\ref{ass:2homogeneous}-\eqref{subass:sard} holds. If the Wasserstein gradient flow of $F$ converges in $W_2$ to $\mu_\infty$, then $\mu_\infty$ is a global minimizer of $F$.
\end{proposition}

We display on Figure~\ref{fig:ReLU} particle gradient flows for training a neural network with a single hidden layer and ReLU activation in the classical (non-differentiable) parameterization, with $d=2$ (no regularization). Features are normally distributed, and the ground truth labels are generated with a similar network with $m_0=4$ neurons. The particle gradient flow is ``integrated'' with mini-batch SGD and the particles are initialized on a small centered sphere.

\begin{figure}[h]
\centering
\begin{subfigure}{0.32\textwidth}
\includegraphics[scale=0.44,trim=0.2cm 0.2cm 0.2cm 0.2cm,clip]{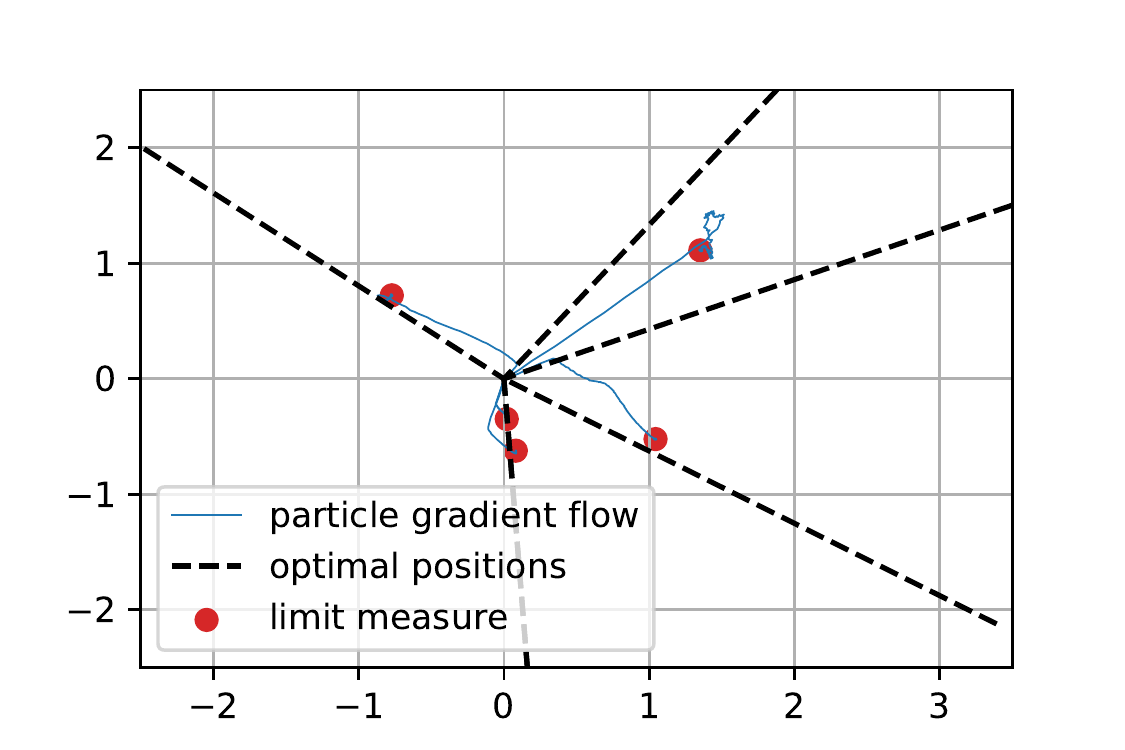}
\end{subfigure}
\begin{subfigure}{0.32\textwidth}
\includegraphics[scale=0.44,trim=0.2cm 0.2cm 0.2cm 0.2cm,clip]{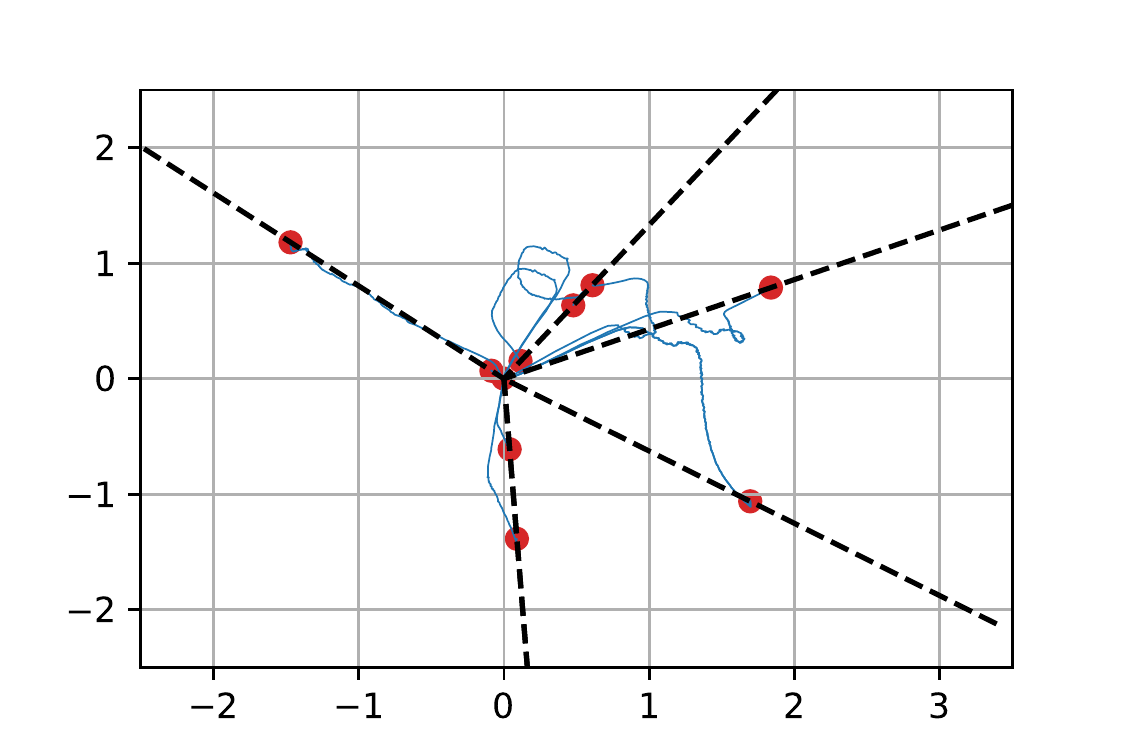}
\end{subfigure}
\begin{subfigure}{0.32\textwidth}
\includegraphics[scale=0.44,trim=0.2cm 0.2cm 0.2cm 0.2cm,clip]{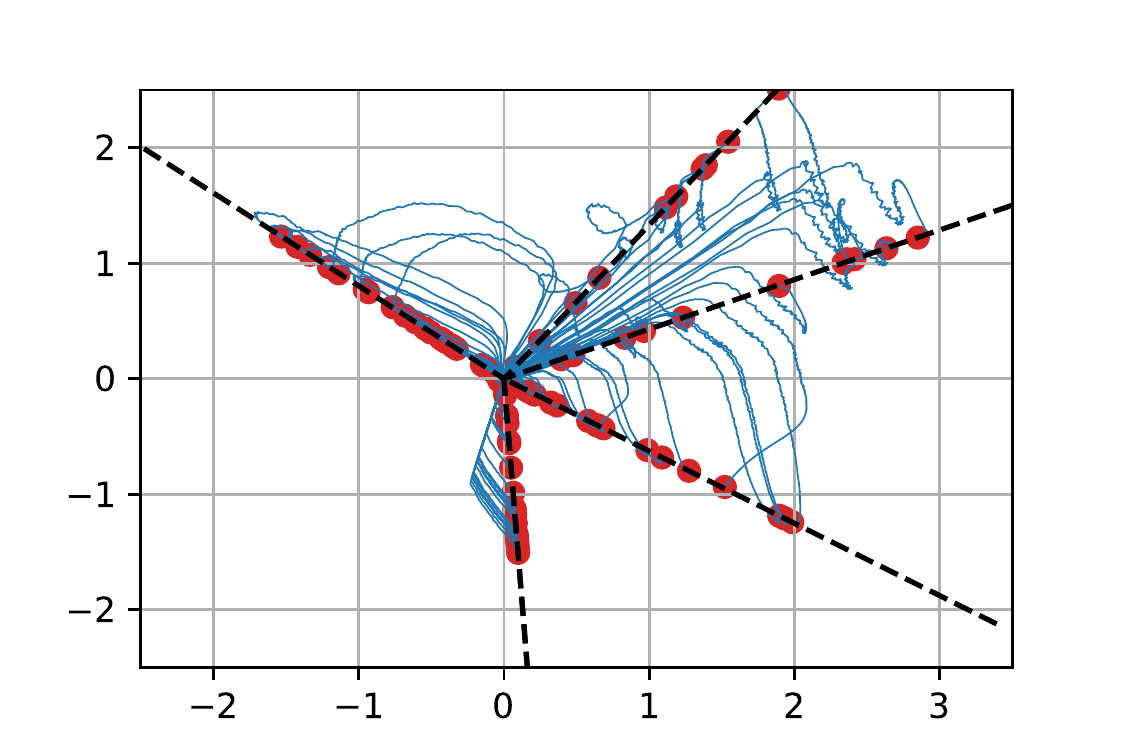}
\end{subfigure}
\caption{Training a neural network with ReLU activation. Failure with $5$ particles (a.k.a.\ neurons), success with $10$ and $100$ particles. We show the trajectory of $\vert w(t)\vert \cdot \theta(t) \in \R^2$ for each particle (an animated plot of this particle gradient flow can be found in Appendix~\ref{app:details}).}
\label{fig:ReLU}
\end{figure}

\subsection{Empirical particle-complexity}
Since our convergence results are non-quantitative, one might argue that similar---and much simpler to prove---asymptotical results hold for the method of distributing particles on the whole of $\Theta$ and simply optimizing on the weights, which is a convex problem. Yet, the comparison of the particle-complexity shown in Figure~\ref{fig:particlecomplexity} stands strongly in favor of particle gradient flows.
While exponential particle-complexity is unavoidable for the convex approach, we observed on several synthetic problems that particle gradient descent only needs a slight over-parameterization $m>m_0$ to find global minimizers within optimization error (see details in Appendix~\ref{app:details}).

\begin{figure}[h]
\centering
\begin{subfigure}{0.01\textwidth}
\rotatebox{90}{\hspace{0.5cm}\tiny{Excess loss at convergence}}
\end{subfigure}
\begin{subfigure}{0.32\textwidth}
\includegraphics[scale=0.46,trim=0.8cm 0.2cm 0.8cm 0.2cm,clip]{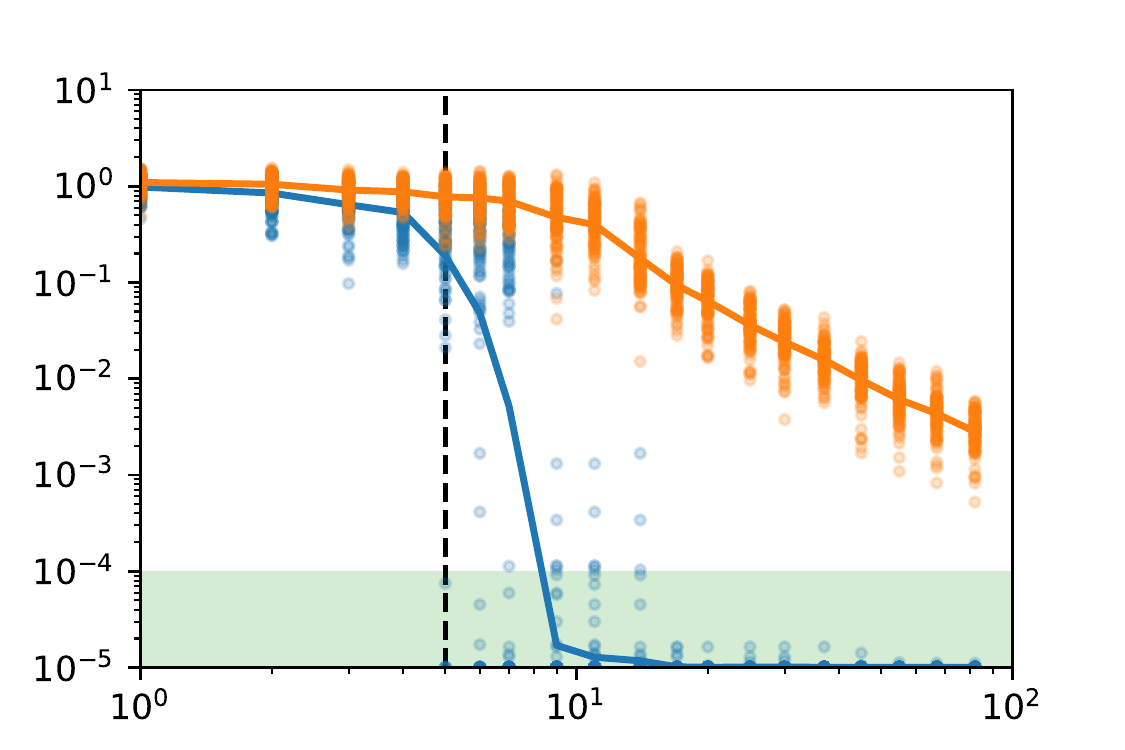}
\caption{Sparse deconvolution ($d=1$)}
\end{subfigure}
\begin{subfigure}{0.32\textwidth}
\includegraphics[scale=0.46,trim=0.8cm 0.2cm 0.8cm 0.2cm,clip]{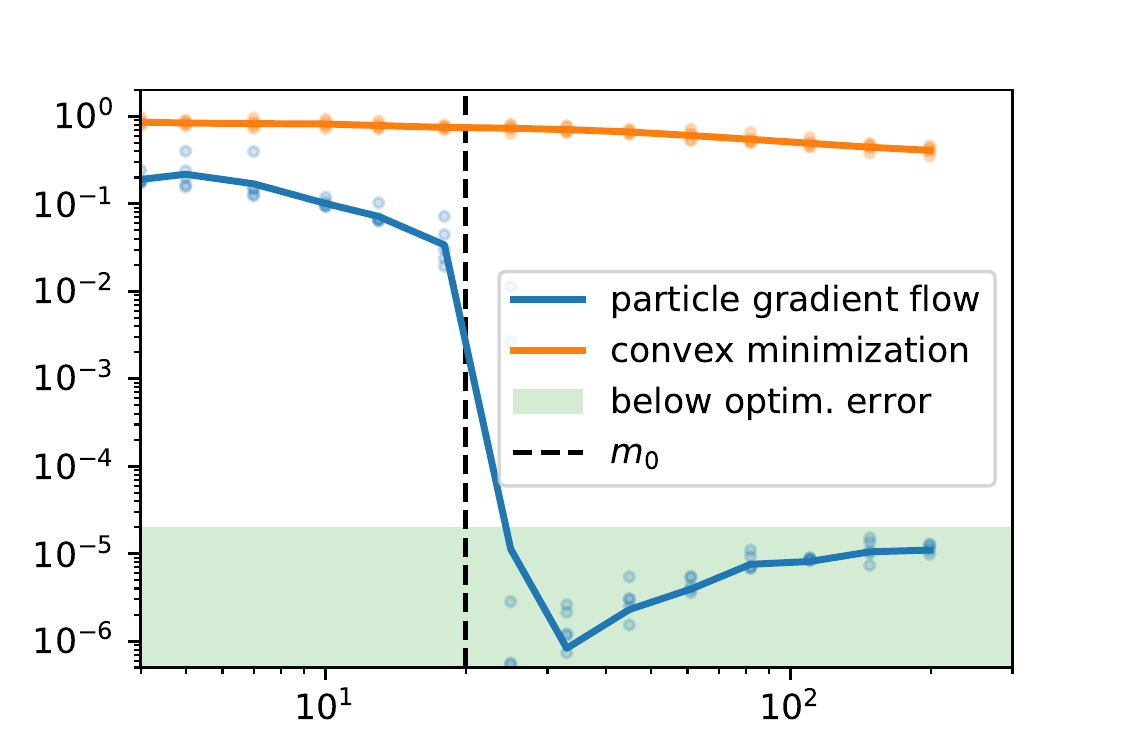}
\caption{ReLU activation ($d=100$)}
\end{subfigure}
\begin{subfigure}{0.32\textwidth}
\includegraphics[scale=0.46,trim=0.8cm 0.2cm 0.8cm 0.2cm,clip]{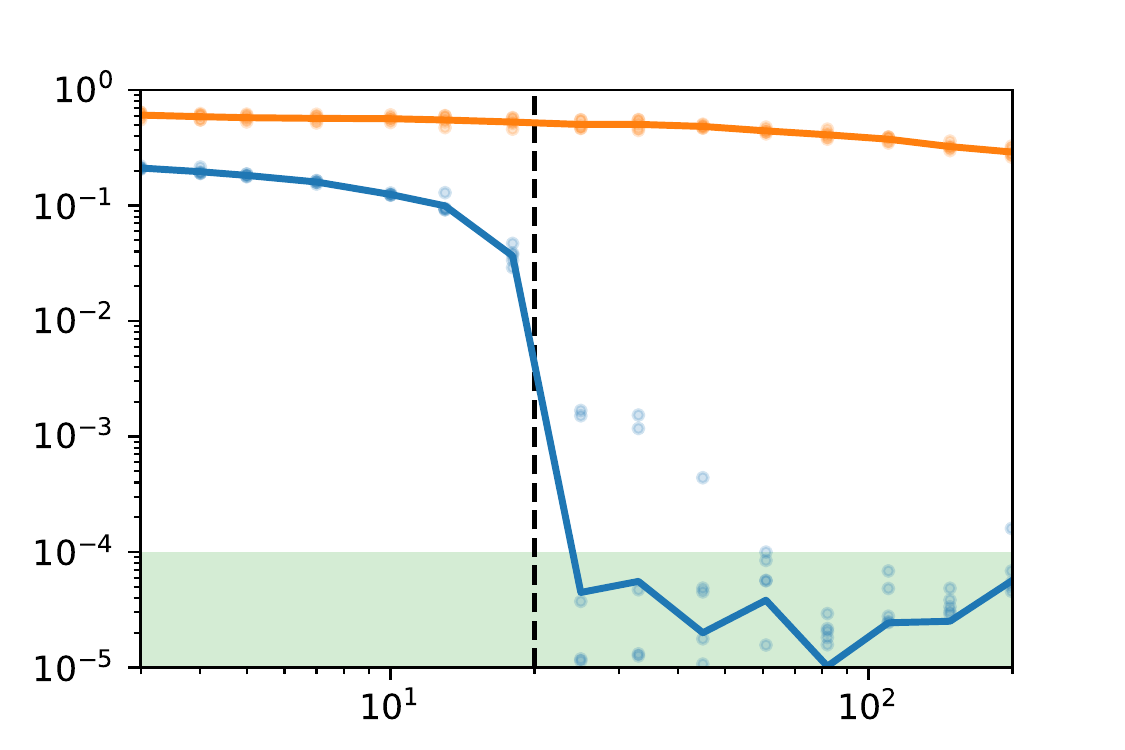}
\caption{Sigmoid activation ($d=100$)}
\end{subfigure}
\caption{Comparison of particle-complexity for particle gradient flow and convex minimization on a fixed grid: excess loss at convergence vs.\ number of particles. Simplest minimizer has $m_0$ particles.}
\label{fig:particlecomplexity}
\end{figure}

\section{Conclusion}
We have established asymptotic global optimality properties for a family of non-convex gradient flows. These results were enabled by the study of a Wasserstein gradient flow: this object simplifies the handling of many-particle regimes, analogously to a mean-field limit.  The particle-complexity to reach global optimality turns out very favorable on synthetic numerical problems. This confirms the relevance of our qualitative results and calls for quantitative ones  that would further exploit the properties of such particle gradient flows. Multiple layer neural networks are also an interesting avenue for future research.

\subsubsection*{Acknowledgments}
We acknowledge supports from grants from R\'egion Ile-de-France and the European Research Council (grant SEQUOIA 724063).

\bibliography{LC}
\bibliographystyle{plain}
\normalsize

\newpage

\section*{Supplementary material}
Supplementary material for the paper: ``On the Global Convergence of Gradient Descent for Over-parameterized Models using Optimal Transport'' authored by L\'ena\"ic Chizat and Francis Bach (NIPS 2018).

This appendix is organized as follows:
\begin{itemize}
\item Appendix~\ref{app:intro}: Introductory facts
\item Appendix~\ref{app:WGF}: Many-particle limit and Wasserstein gradient flow
\item Appendix~\ref{app:global}: Convergence to global minimizers
\item Appendix~\ref{app:case}: Case studies and numerical experiments
\end{itemize}
\appendix

% !TEX root = ../manyparticle.tex

\section{Introductory facts} \label{app:intro}

\subsection{Tools from measure theory}
In this paper, the term \emph{measure} refers to a \emph{finite} signed measure on $\R^d$, $d\geq 1$, endowed with its Borel $\sigma$-algebra. We write $\M(X)$ for the set of such measures concentrated on a measurable set $X\subset \R^d$. Hereafter, we gather some concepts and facts from measure theory that are used in the proofs.

\paragraph{Variation of a signed measure.} The Jordan decomposition theorem~\cite[Cor. 4.1.6]{cohn1980measure} asserts that any finite signed measure $\mu \in \M(\R^d)$ can be decomposed as $\mu=\mu_+-\mu_-$ where $\mu_+,\mu_-\in \M_+(\R^d)$. If $\mu_+$ and $\mu_-$ are chosen with minimal total mass, the \emph{variation} of $\mu$ is the nonnegative measure $\vert \mu \vert := \mu_+ + \mu_-$ and $\vert \mu \vert (\R^d)$ is the \emph{total variation norm} of $\mu$.

\paragraph{Support and concentration set.} The \emph{support} $\spt \mu$ of a measure $\mu \in \M_+(\R^d)$ is the complement of the largest open set of measure $0$, or, equivalently, the set of points which neighborhoods have positive measure. We say that $\mu$ is \emph{concentrated} on a set $S\subset \R^d$ if the complement of $S$ is included in a measurable set of measure $0$. In particular, $\mu$ is concentrated on $\spt \mu$. 

\paragraph{Pushforward.}
Let $X$ and $Y$ be measurable subsets of $\R^d$ and let $T:X\to Y$ be a measurable map. To any measure $\mu \in \M(X)$ corresponds a measure $T_\# \mu \in \M(Y)$ called the \emph{pushfoward} of $\mu$ by $T$. It is defined as $T_\#\mu (B) = \mu(T^{-1}(B))$ for all measurable set $B\subset Y$
and corresponds to the distribution of the ``mass" of $\mu$ after it has been displaced by the map $T$. It satisfies $\int_Y \varphi\, \d (T_\#\mu) = \int_X \varphi \circ T \, \d \mu$ whenever $\varphi: Y\to \R$ is a measurable function such that $\varphi \circ T$ is $\mu$-integrable~\cite[Prop. 2.6.8]{cohn1980measure}. In particular, with a projection map $\pi^i: (x_1,x_2,\dots)\mapsto x_i$, the pushforward $\pi^i_\# \mu$ is the \emph{marginal} of $\mu$ on the $i$-th factor.

\paragraph{Weak convergence and Bounded Lipschitz norm.} We say that a sequence of measures $\mu_n \in \M(\R^d)$ \emph{weakly} (or \emph{narrowly}) converges to $\mu$ if, for all continuous and bounded function $\varphi: \R^d \to \R$ it holds $\int \varphi \d \mu_n \to \int \varphi \d \mu$. For sequences which are bounded in total variation norm, this is equivalent to the convergence in Bounded Lipschitz norm. The latter is defined, for $\mu \in \M(\R^d)$, as 
\begin{equation}\label{eq:BL}
\Vert \mu \Vert_\BL := \sup \left\{  \int \varphi \, \d \mu\; ;\; \varphi: \R^d \to \R, \; \Lip(\varphi)\leq 1,\; \Vert \varphi \Vert_\infty \leq 1 \right\}
\end{equation}
where $\Lip(\varphi)$ is the smallest Lipschitz constant of $\varphi$ and $\Vert \cdot \Vert_\infty$ the supremum norm. 

\paragraph{Wasserstein metric.}  
The $p$-Wasserstein distance between two probability measures $\mu,\nu \in \P(\R^d)$ is defined as
\begin{equation*}
W_p(\mu,\nu):= \left( \min \int \vert y-x \vert^p d \gamma(x,y)\right)^{1/p}
\end{equation*}
where the minimization is over the set of probability measures $\gamma \in \P(\R^d \times \R^d)$ such that the marginal on the first factor $\R^d$ is $\mu$ and is $\nu$ on the second factor. The set of probability measures with finite second moments endowed with the metric $W_2$ is a complete metric space that we denote $\P_2(\R^d)$. A sequence $(\mu_m)_m$ converges in $\P_2(\R^d)$ iff for all continuous function $\varphi: \R^d\to \R$ with at most quadratic growth it holds $\int \varphi \d \mu_m \to \int \varphi \d \mu$~\cite[Prop. 7.1.5]{ambrosio2008gradient} (this is stronger than weak convergence). Using, respectively, the duality formula for $W_1$~\cite[Eq. (3.1)]{santambrogio2015optimal} and Jensen's inequality, it holds
\[
\Vert \mu-\nu\Vert_\BL \leq W_1(\mu,\nu) \leq W_2(\mu,\nu).
\]

Note that the functional of interest in this article is continuous for the Wasserstein metric. This strong regularity is rather rare in the study of Wasserstein gradient flows.
\begin{lemma}[Wasserstein continuity of $F$]\label{lem:Fcontinuous}
Under Assumptions~\ref{ass:regularity}, the function $F$ is continuous for the Wasserstein metric $W_2$.
\end{lemma}
\begin{proof}
Let $(\mu_m)_m,\mu \in \P_2(\Omega)$ be such that $W_2(\mu_m,\mu)\to 0$. By Assumption~\ref{ass:regularity}-(iii)-(c), $\Vert \Phi\Vert$ and $\vert V\vert$ have at most quadratic growth. It follows  $\int V \d \mu_m \to \int V \d \mu$ and since, by the properties of Bochner integrals~\cite[Prop. E.5]{cohn1980measure}, it holds
$
\Vert \int \Phi \d \mu_m - \int \Phi \d \mu \Vert \leq \int \Vert \Phi \Vert\d (\mu_m-\mu),
$
we also have $\int \Phi \d \mu_m \to \int \Phi \d \mu$ strongly in $\F$.
As $R$ is continuous in the strong topology of $\F$, it follows $F(\mu_m)\to F(\mu)$.
\end{proof}

\subsection{Lifting to the space of probability measures}\label{app:lifting}
Let us give technical details about the lifting introduced in Section~\ref{subsec:lifting} that allows to pass from a problem on the space of signed measures on $\Theta \subset \R^{d-1}$ (the minimization of $J$ defined in~\eqref{eq:introproblem}) to an equivalent problem on the space of probability measures on a bigger space $\Omega \subset \R^{d}$ (the minimization of $F$ defined in~\eqref{eq:mainproblem}). 

\paragraph{Homogeneity.} 
We recall that a function $f$ from $\R^d$ to a vector space is said \emph{positively $p$-homogeneous}, with $p\geq 0$ if for all $u\in \R^d$ and $\lambda>0$ it holds $f(\lambda u)=\lambda^p f(u)$. We often use without explicit mention the properties related to homogeneity such as the fact that the (sub)-derivative of a positively $p$-homogeneous function is positively $(p-1)$-homogeneous and, for $f$ differentiable (except possibly at $0$), the identity $u\cdot \nabla f(u) = p f(u)$ for $u \neq 0$.

\subsubsection{The partially $1$-homogeneous case}
We take $\Omega := \R \times \Theta$, $\Phi(w,\theta)= w \cdot \phi(\theta)$ and $V(w,\theta)= \vert w\vert\tilde V(\theta)$ for some continuous functions $\phi:\Theta\to \F$ and $\tilde V: \Theta\to \R_+$. This setting covers the lifted problems mentioned in Section~\ref{subsec:lifting}. We first show that $F$ can be indifferently minimized over  $\M_+(\Omega)$ or over $\P(\Omega)$, thanks to the homogeneity of $\Phi$ and $V$ in the variable $w$.
\begin{proposition}
For all $\mu \in \M_+(\Omega)$, there is $\nu \in \P(\Omega)$ such that $F(\mu)=F(\nu)$.
\end{proposition}
\begin{proof}
If $\vert\mu\vert(\Omega) = 0$ then $F(\mu)=0=F(\delta_{(0,\theta_0})$ where $\theta_0$ is any point in $\Theta$. Otherwise, we define the map $T:(w,\theta)\mapsto (\vert\mu\vert(\Omega)\cdot w,\theta)$ and the probability measure $\nu := T_\#(\mu/\vert\mu\vert(\Omega)) \in \P(\Omega)$, which satisfies $F(\nu)=F(\mu)$.
\end{proof}
We now introduce a projection operator $ h^1: \M_+(\Omega)\to \M(\Theta)$ that is adapted to the partial homogeneity of $\Phi$ and $V$. It is defined by $h^1(\mu)(B) = \int_\R w  \mu(d w,B)$ for all $\mu \in \P(\Omega)$ and measurable set $B\subset \Theta$ or, equivalently, by the property that for all continuous and bounded test function $\varphi : \Theta \to \R$,
\[
\int_\Theta \varphi(\theta) \d h^1(\mu)(\theta) = \int_{\R\times \Theta} w \varphi(\theta) d \mu(w,\theta).
\]
This operator is well defined whenever $(w,\theta)\mapsto w$ is $\mu$-integrable.
\begin{proposition}[Equivalence under lifting]\label{prop:liftsigned}
It holds $\M(\Theta)\subset h^1(\P(\Omega))= h^1(\M_+(\Omega)) $. For a regularizer $G$ on $\M(\Theta)$ of the form $G(\mu) = \inf_{\nu \in h^{-1}(\mu)} \int_{\Omega} Vd \nu$, it holds
$\inf_{\nu \in \M(\Theta)} J(\nu) =  \inf_{\mu \in \M_+(\Omega)}  F(\mu)$. If the infimum defining $G$ is attained and if $\nu\in \M(\Theta)$ minimizes $J$, then there exists $\mu \in h^{-1}(\nu)$ that minimizes $F$ over $\M_+(\Omega)$.
\end{proposition}

\begin{proof}
A signed measure $\nu\in \M(\Theta)$ can be expressed as $\nu = f \sigma$ where $\sigma \in \P(\Theta)$ and $f: \Theta \to \R \in L^1(\sigma)$ (take for instance $\sigma$ the normalized variation of $\mu$ if $\vert\mu\vert(\Theta)>0$). The measure 
\begin{equation}\label{eq:h1preimage}
\mu := ( f\times \id)_\# \sigma
\end{equation}
belongs to $\P(\Omega)$ and satisfies $h^1(\nu)= \mu$. This proves that $h^1(\P(\Omega))$ is surjective. It is clear by the definition of $h^1$ that for all $\mu \in h^{-1}(\nu)$, it holds $\int \Phi\, \d \mu = \int \phi \, \d \nu$ hence $F(\mu)\geq J(\nu)$, with equality when $\mu$ is  the minimizer in the definition of $G$.
\end{proof}

The class of regularizer considered in Proposition~\ref{prop:liftsigned} includes the total variation norm.
\begin{proposition}[Total variation]
Let $V(w,\theta)=\vert w \vert$.
For $\mu \in \M(\Theta)$, it holds $\int V d\mu \geq \vert h^1(\mu)\vert(\Theta)$ with equality if, for instance, $\mu$ is a lift of $h^1(\mu)$ of the form~\eqref{eq:h1preimage}.
\end{proposition}
\begin{proof}
Let $\mu\in \P(\Omega)$ and $\nu = h^1(\mu)$. We define $\tilde \nu_+ := \int_{\R_+} w\,\mu(w,\cdot)$ and $\tilde \nu_- := -\int_{\R_-} w\,\mu(w,\cdot)$. Clearly, $\nu = \tilde \nu_+ - \tilde \nu_-$ and by the definition of the total variation of a signed measure, $\vert \nu\vert(\Theta) = \vert \nu_+\vert(\Theta)+\vert \nu_-\vert(\Theta)  \leq  \vert \tilde \nu_+\vert(\Theta)+\vert \tilde \nu_-\vert(\Theta)  = \int V \d \mu$. There is equality whenever $\spt \tilde \nu_+\cap \spt \tilde \nu_-$ has $\vert \nu\vert$-measure $0$ (see~\cite[Cor. 4.1.6]{cohn1980measure}), a condition which is satisfied by the lift in~\eqref{eq:h1preimage}.
\end{proof}

\subsubsection{The $2$-homogeneous case}
Another structure that is studied in this paper is when $\Phi$ and $V$ are defined on $\R^d$ and are positively $2$-homogeneous. In this case, the role played by $\Theta$ is the previous section is played by the unit sphere $\S^{d-1}$ of $\R^d$. We could again make links between $F$ (defined as in Eq.~\eqref{eq:mainproblem}) and a functional on nonnegative measures on the sphere (playing the role of $J$) but here we will limit ourselves to defining the projection operator relevant in this setting. It is $h^2 : \M_+(\R^d)\to \M_+(\S^{d-1})$ characterized by the relationship, for all continuous and bounded function $\varphi : \S^{d-1} \to \R$ (with the convention $\phi(0/0)=0$):
\[
\int_{\S^{d-1}} \varphi(\theta) \d h^2(\mu)(\theta) = \int_{\R^d} \vert u\vert^2 \varphi(u/\vert u\vert) d \mu(u).
\]
This operator is well-defined iff $\mu$ has finite second order moments.

% !TEX root = ../manyparticle.tex

\section{Many-particle limit and Wasserstein gradient flow}\label{app:WGF}

\subsection{Proof of Proposition~\ref{eq:classicalGF}}
%\begin{proof}
As the sum of a continuously differentiable and a semiconvex function, $F_m$ is locally semiconvex and the existence of a unique gradient flow on a maximal interval $[0,T[$ with the claimed properties is standard, see~\cite[Sec. 2.1]{santambrogio2017euclidean}.  Now, a general property of gradient flows is that for a.e $t\in \R_+, u \in \Omega$, the derivative is (minus) the subgradient of minimal norm. This leads to the explicit formula involving the velocity field with pointwise minimal norm:
\begin{align*}
v_t(u) &= \arg\min \left\{ \vert v \vert^2 \;;\; \tilde v_t(u)  - v \in \partial V(u) \right\}\\
&= \tilde v_t(u) - \arg \min \left\{ \vert \tilde v_t(u) -z\vert^2 \;;\; z \in \partial V(u)\right\} \\
&= (\mathrm{id} -\mathrm{proj}_{\partial V(u)})(\tilde v_t(u)).
\end{align*}

 In the specific case of gradient flows of lower bounded functions, we can derive estimates that imply that $T=\infty$ (even if $F_m$ is not globally semiconvex). Indeed, for all $t>0$, it holds
\begin{align*}
F_m(\mathbf u(0))-F_m(\mathbf u(t)) &= -\int_0^t \frac{d}{ds}F_m(\mathbf u(s))\d s %\\
= \frac1m \int_0^t \vert \mathbf u'(s)\vert^2\d s \geq \frac{t}{m}\left( \int_0^t \vert\mathbf u'(s)\vert\d s\right)^2
\end{align*}
by Jensen's inequality. Since $F_m$ is lower bounded, this proves that the gradient flow has bounded length on bounded time intervals. By compactness, if $T$ was finite then $\mathbf u(T)$ would exist, thus contradicting the maximality of $T$, hence $T=\infty$ and the gradient flow is globally defined.
%\end{proof}

\subsection{Link between classical and Wasserstein gradient flows}
We first give a rigorous definition of the continuity equation which appear in the definition of Wasserstein gradient flows (Definition~\ref{def:WassersteinGF}).
\paragraph{Continuity equation.}
Considerations from fluid mechanics suggest that if a time dependent distribution of mass $(\mu_t)_t$ is displaced under the action of a velocity field $(v_t)_t$, then the \emph{continuity equation} is satisfied: $\partial_t \mu_t = - \mathrm{div} (v_t\mu_t)$. 
For distributions which do not have a smooth density, this equation should be understood distributionally, which means that for all smooth test functions $\varphi: {]0,\infty[\times \R^d}$ with compact support, it holds
\[
\int_0^\infty \int_{\R^d} \left(\partial_t \varphi_t(u) +\nabla_u \varphi_t(u)\cdot v_t(u) \right)\d \mu_t(u)\d t = 0.
\]
The integrability condition $\int_0^{t_0}\int_{\R^d}\vert v_t(u)\vert\d\mu(u)\d t<\infty$ for all $t_0<T$ should also hold.

As we show now, there is a precise link between classical and Wasserstein gradient flow (Definitions~\ref{def:classicalGF} and~\ref{def:WassersteinGF}). This is a simple result but might be instructive for readers who are not familiar with the concept of distributional solutions of partial differential equations.
\begin{proposition}[Atomic solutions of Wasserstein gradient flow]\label{prop:atomicWGF}
If $\mathbf u:\R_+\to \Omega^m$ is a classical gradient flow for $F_m$ in the sense of Definition~\ref{def:classicalGF}, then $t\mapsto \mu_{m,t} := \frac1m \sum_{i=1}^m \delta_{\mathbf u_i(t)}$ is a Wasserstein gradient flow of $F$ in the sense of Definition~\ref{def:WassersteinGF}.
\end{proposition}
\begin{proof}
Let us call $v_t$ the velocity vector field defined in~\eqref{eq:velocity}. In it easy to see that $t\mapsto \mu_{m,t}$ is absolutely continuous for $W_2$ and for any smooth function with compact support $\varphi : {]0,\infty[}\times \R^d\to \R$, we have
\begin{align*}
0 &= \frac1m \sum_{i=1}^m \int_{\R_+} \frac{d}{dt} \varphi_t(\mathbf u_i(t))\,\d t \\
&=\frac1m \sum_{i=1}^m  \int_{\R_+} \left( \partial_t \varphi_t(\mathbf u_i(t)) +\nabla_u \varphi_t(\mathbf u_i(t)) \cdot v_t(\mathbf u_i) \right)\d t\\ 
&=
\int_{\R_+}\int_{\Omega} \left( \partial_t \varphi_t(u) + \nabla_u \varphi_t(u)\cdot v_t(u)\right) \d \mu_{m,t}\d t 
\end{align*}
which precisely means that $(\mu_{m,t})_{t\geq 0}$ is a distributional solution to~\eqref{eq:gradientflow}.
\end{proof}
Note that $(\mu_{m,t})_t$ has the same number of atoms throughout the dynamic. In particular, if no minimizer of $F$ is an atomic measure with at most $m$ atoms, then $(\mu_{m,t})_t$ is guaranteed to \emph{not} converge to a minimizer.

\subsubsection{Properties of the Wasserstein gradient flow (proof of Proposition~\ref{prop:uniquenessWGF})}\label{app:WGFproperties}
In this section, we use the general theory of Wasserstein gradient flows developed in~\cite{ambrosio2008gradient} to prove existence and uniqueness of Wasserstein gradient flows as claimed in Proposition~\ref{prop:uniquenessWGF}, under Assumptions~\ref{ass:regularity}.  The ``existence'' part of the proof is in fact redundant with Theorem~\ref{th:manyparticle} which provides with another constructive proof. We recall that $F'(\mu): \Omega \to \R$ is defined as
\[
F'(\mu)(u) = \langle R'(\tint \Phi \d\mu), \Phi(u) \rangle + V(u)
\]
and that the field of subgradients of minimal norm of $F'(\mu)$ has an explicit formula given in~\eqref{eq:velocity}.
Our strategy is to use, as an intermediary step, the Wasserstein gradient flows for the family of functionals $F^{(r)}:\P_2(\Omega)\to \R$ defined, for $r>0$, as
\[
F^{(r)}(\mu) = 
\begin{cases}
F(\mu)&\text{if $\mu(Q_r)=1$},\\
\infty &\text{otherwise}
\end{cases}
\]
where $(Q_r)_{r>0}$ is the nested family of subsets of $\R^d$ that appear in Assumptions~\ref{ass:regularity}.
These ``localized'' functionals have nice properties in the Wasserstein geometry, as shown in Lemma~\ref{lem:wassersteinpropertiesJr}. For $r>0$, we say that $\gamma\in \P(\Omega\times \Omega)$ is an \emph{admissible transport plan} if both its marginals are concentrated on $Q_r$ and have finite second moments. The transport cost associated to $\gamma$ is denoted $C_p(\gamma) := \left( \int \vert y-x\vert^p\d\gamma(x,y)\right)^{1/p}$ for $p\geq 1$, and we introduce the quantities
\begin{align*}
\Vert d\Phi\Vert_{\infty,r} &= \sup_{u \in Q_r} \Vert d\Phi_u\Vert & L_{d\Phi} &= \sup_{\substack{u,\tilde u \in Q_r\\u\neq \tilde u}}\frac{\Vert d\Phi_{\tilde u}-d\Phi_u\Vert}{\vert \tilde u-u\vert} \\
\Vert dR\Vert_{\infty,r} &= \sup_{f\in \F_r}\Vert dR_f\Vert & L_{dR} &=\sup_{\substack{f,g\in \F_r\\ f\neq g}} \frac{\Vert dR_f-dR_g \Vert}{\Vert f-g\Vert} 
\end{align*}
where $\F_r \coloneqq \{ \tint \Phi \d\mu\;;\; \mu \in \P(\Omega),\, \mu(Q_r)=1\}$ is bounded in $\F$. Those quantities are finite for all $r>0$ under Assumptions~\ref{ass:regularity}.  For the sake of clarity, we set $V=0$ in the next lemma, and focus on the \emph{loss} term, which is a new object of study. The term involving $V$, well-studied in the theory of Wasserstein gradient flows (see~\cite[Prop. 10.4.2]{ambrosio2008gradient}), is incorporated later.

\begin{lemma}[Properties of $F^{(r)}$ in Wasserstein geometry]\label{lem:wassersteinpropertiesJr}
Under Assumptions~\ref{ass:regularity}, suppose that $V=0$. For all $r>0$,  $F^{(r)}$ is proper and continuous for $W_2$ on its closed domain. Moreover, 
\begin{enumerate}[(i)]
\item there exists $\lambda_r>0$ such that for all admissible transport plan $\gamma$, considering the transport interpolation $\mu^\gamma_t\coloneqq ((1-t)\pi^1 +t\pi^2)_\# \gamma$, the function $t\mapsto F(\mu^\gamma_t)$ is differentiable with a $\lambda_r C_2^2(\gamma)$-Lipschitz derivative;
\item for $\mu$ concentrated on $Q_r$, a velocity field $v\in L^2(\mu,\R^d)$ satisfies, for any admissible transport plan $\gamma$ with first marginal $\mu$, 
\[
F(\pi^2_\#\gamma) \geq F(\mu) + \int v(u)\cdot (\tilde u - u)\d\gamma(u,\tilde u) + o(C_2(\gamma))
\] 
if and only if $v(u) \in \partial (F'(\mu)+ \iota_{Q_r})(u)$ for $\mu$-almost every $u\in \Omega$, where $\iota_{Q_r}$ is the convex function on $\Omega$ that is worth $0$ on $Q_r$ and $\infty$ outside.
\end{enumerate}
\end{lemma}

\begin{proof}
First, it is clear that $F$ is proper because $F^{(r)}(\delta_{u_0})=R(\Phi(u_0))$ is finite whenever $u_0\in Q_r$. It is moreover continuous (see Lemma~\ref{lem:Fcontinuous}) on its closed domain $\{ \mu \in \P_2(\Omega)\;;\; \mu(Q_r)=1\}$.

\paragraph{Proof of (i).} 
Let us denote $h(t):= F^{(r)}(\mu^\gamma_t)$. Since $dR$ and $d\Phi$ are Lipschitz on $\F_r$ and $Q_r$ respectively, $h(t)$ is differentiable with
\begin{align*}
h'(t) = \frac{d}{dt}F^{(r)}(\mu^\gamma_t) = \left\langle R'(\tint\Phi\d \mu^\gamma_t), \tint d\Phi_{(1-t)x+ty}(y-x) \d\gamma(x,y) \right\rangle.
\end{align*}
In particular, we can differentiate $t\mapsto \int \Phi\d\mu^\gamma_t = \int \Phi((1-t)x +ty)\d\gamma(x,y)$ because all $(\mu^\gamma_t)_t$ are supported on $Q_r$ where $d\Phi$ is uniformly bounded and Bochner integrals admit a dominated convergence theorem~\cite[Thm. E6]{cohn1980measure}. For $0\leq t_1<t_2<1$, we have the bounds
\[
\vert h'(t_2)-h'(t_1)\vert \leq (I)+ (II)
\]
where, one the one hand,
\begin{align*}
(I)&= \left\vert \left\langle R'\left(\tint\Phi\d \mu_{t_2}\right)-R'\left(\tint\Phi\d \mu_{t_1}\right), \tint d\Phi_{(1-t_2)x+t_2y}(y-x) \d\gamma(x,y) \right\rangle  \right\vert \\
&\leq \left[ L_{dR} \cdot \Vert d\Phi\Vert_{\infty,r}\cdot \vert t_2-t_1\vert \cdot C_1(\gamma)\right] \cdot \left[\Vert d\Phi\Vert_{\infty,r} \cdot C_1(\gamma)\right]\\
&\leq L_{dR} \cdot \Vert d\Phi\Vert_{\infty,r}^2\cdot C_2^2(\gamma)\cdot \vert t_2-t_1\vert 
\end{align*}
where we used H\"older's inequality to obtain $C_1^2(\gamma)\leq C_2^2(\gamma)$ is the last line. On the other hand,
\begin{align*}
(II) &= \left\vert \left\langle  R'\left(\tint\Phi\d \mu_{t_1}\right), \tint \left[ d\Phi_{(1-t_2)x+t_2y}-d\Phi_{(1-t_1)x+t_1y}\right](y-x)\d\gamma(x,y)  \right\rangle\right\vert \\
&\leq L_{d\Phi}\cdot \Vert dR\Vert_{\infty,r}\cdot C_2^2(\gamma)\cdot \vert t_2-t_1\vert.
\end{align*}
As a consequence, $h'$ is $\lambda_r\cdot C_2^2(\gamma)$ Lipschitz with $\lambda_r = L_{dR} \Vert d\Phi\Vert_{\infty,r}^2  + L_{d\Phi} \Vert dR\Vert_{\infty,r}$. In particular, using the notions defined in~\cite{ambrosio2008gradient} , $F^{(r)}$ is $(-\lambda_r)$-geodesically semiconvex. Remark that these bounds may explode when $r$ goes to infinity: this explains why we work with measures supported on $Q_r$.

\paragraph{Proof of (ii).} The proof is similar, with the difference that this property is a local one. 
We have the first-order Taylor expansions, for $u,\tilde u \in Q_r$ and $f,g \in \F_r$,
\begin{align*}
\Phi(\tilde u) &= \Phi(u) + d\Phi(\tilde u-u) +M(u,\tilde u) \\
R(g) &= R(f) + \langle R'(f),g-f\rangle + N(f,g)
\end{align*}
where the remainders $M$ and $N$ satisfy $\Vert M(u,\tilde u)\Vert\leq \frac12 L_{d\Phi} \cdot \vert \tilde u-u\vert^2$ and $\Vert N(f,g) \Vert \leq \frac12 L_{dR}\cdot \Vert g-f\Vert^2$. 
We denote by $\mu$ and $\nu$ the first and second marginals of $\gamma$, assume that they are both concentrated on $Q_r$, and obtain, by composition, the Taylor expansion
\begin{align*}
F^{(r)}(\nu) = F^{(r)}(\mu) + \left\langle R'(\tint \Phi\d\mu),\tint d\Phi_{u}(\tilde u-u)\d\gamma(u,\tilde u) \right\rangle + (I) + (II)
\end{align*}
where $(I) =   \left\langle R'(\tint \Phi\d\mu),\tint M(u,\tilde u)\d\gamma(u,\tilde u) \right\rangle$, so
\[
\vert (I)\vert \leq \frac12 \Vert dR\Vert_{\infty,r}\cdot L_{d\Phi} \cdot C_2^2(\gamma) = o(C_2(\gamma))
\]
and $(II) = N(\tint \Phi\d\mu, \tint \Phi \d\nu)$, so
\begin{align*}
 \vert (II)\vert &\leq  \frac12 L_{dR} \cdot \Vert \tint d\Phi_u(\tilde u-u)\d\gamma(u,\tilde u) + \tint M(u,\tilde u)\d\gamma(u,\tilde u) \Vert^2\\
 &\leq \frac12 L_{dR} \cdot \left( \Vert d\Phi\Vert_{\infty,r} \cdot C_1(\gamma) +  \frac12 \cdot L_{d\Phi} \cdot C_2^2(\gamma) \right)^2 = o(C_2(\gamma))
\end{align*}
where we used H\"older's inequality for the bound $C_1^2(\gamma)\leq C_2^2(\gamma)$. As a consequence,
\[
F^{(r)}(\nu) = F^{(r)}(\mu) + \int \langle R'(\tint \Phi\d\mu),d\Phi_{u}(\tilde u-u)\rangle \d\gamma(u,\tilde u)  + o(C_2(\gamma))
\]
and remember that the $j$-th component of $\nabla F'(\mu)$ is $u\mapsto \langle R'(\tint \Phi\d\mu),d\Phi_u({e_j})\rangle$ where $e_j$ is the $j$-th vector of the canonical basis of $\R^d$. This completely characterizes a velocity field satisfying (ii) on the interior of $Q_r$. On the boundary of $Q_r$, there is more freedom in the choice of $v(u)$ since $\pi^2_\# \gamma$ is constrained to be supported on $Q_r$ so $v(u)-\nabla F'(\mu)(u)$ can live in the normal cone of $Q_r$ at $u$, which is the set $\partial \iota_{Q_r}(u)$. The condition thus relaxes as $v(u) \in \partial (F'(\mu) + \iota_{Q_r})(u)$.
\end{proof}

The previous properties are sufficient to guarantee that Wasserstein gradient flows for the functionals $F^{(r)}$ are well defined.
\begin{lemma}\label{lem:WGF_Fr}
Under Assumptions~\ref{ass:regularity}, there exists a unique Wasserstein gradient flow for $F^{(r)}$ starting from any $\mu_0 \in \P_2(\Omega)$ concentrated on $Q_{r}$, i.e.\ a curve $(\mu^{(r)}_t)_{t\geq 0}$, continuous in $\P_2(\Omega)$, that solves $\partial_t \mu^{(r)}_t +\mathrm{div}(v^{(r)}_t\mu^{(r)}_t)=0$ where, for all $t>0$, $v^{(r)}_t(u) \in \partial (F'(\mu^{(r)}_t)(u) + \iota_{Q_r}(u))$ for $\mu^{(r)}_t$-a.e $u\in \Omega$.
\end{lemma}
\begin{proof}
It is easy to see that if $V$ is $\lambda_V$-semiconvex, then the function $\mu \mapsto \int V \d \mu$ is $\lambda_V$ semiconvex along generalized geodesics (in the sense of \cite[Def. 9.2.4]{ambrosio2008gradient}, see~\cite[Prop. 10.4.2]{ambrosio2008gradient}). Combining with Lemma~\ref{lem:wassersteinpropertiesJr}-(i), we have that  $F^{(r)}$ is $(\lambda_V-\lambda_r)$-semiconvex along generalized geodesics.
Moreover, Lemma~\ref{lem:wassersteinpropertiesJr}-(ii) implies that $F^{(r)}$ admits strong Wasserstein subdifferentials on its domain~\cite[Def 10.3.1]{ambrosio2008gradient} and again, it is an easy adaptation to show that (ii) still holds with a potential term. So the existence of a unique Wasserstein gradient flow characterized as above is guaranteed by~\cite[Thm. 11.2.1]{ambrosio2008gradient}.
\end{proof}

We are in position to prove the well-posedness of Wasserstein gradient flows for the original functional $F$. Notice that, by the characterization in Lemma~\ref{lem:WGF_Fr}, the Wasserstein gradient flows for the functions $F^{(r)}$ all coincide for $r>r_0>0$ on $[0,T]$ if $\mu^{(2r_0)}_t$ is concentrated in $Q_{r_0}$ for all $t\in [0,T]$. Our strategy is thus simply to make sure that for all time $T$, such a $r_0>0$ exists, i.e. to make sure that the support of gradient flows does not grow too fast.

\begin{proof}[Proof of Proposition~\ref{prop:uniquenessWGF}]
Let $r_0$ be such that $\mu_0$ is concentrated on $Q_{r_0}$.
Given Lemma~\ref{lem:WGF_Fr}, for all $r>r_0$, there exists a unique, globally defined, Wasserstein gradient flow $(\mu^{(r)}_t)_{t\geq0}$ for $F^{(r)}$. 
For all $r>r_0$, consider the first exit time from $Q_r$:
\[
t_r \coloneqq \inf \{ t>0\;;\;  \mu^{(2r)}_t (Q_r)<1 \}.
\]
Note that the definition of $t_r$ involves the flow $(\mu_t^{2r})_t$ but in fact, for all $\bar r> r$ and $0\leq t\leq t_r$, it holds $\mu_t^{(2r)}=\mu_t^{(\bar r)}$ by the uniqueness in Lemma~\ref{lem:WGF_Fr}. Thus, if $t_r>0$, we have existence and uniqueness of a Wasserstein gradient flow in the sense of Definition~\ref{def:WassersteinGF} on $[0,t_r]$. It only remains to show that $\lim_{r\to \infty} t_r =\infty$ so that the gradient flow can be defined at all times.

Given the property of $v^{(r)}$ in Lemma~\ref{lem:wassersteinpropertiesJr}-(ii), for all time $0\leq t\leq t_r$, it holds $v^{(r)}_t \in \partial F'(\mu^{(r)}_t)$ in $L^2(\mu^{(r)}_t;\R^d)$. Therefore, using Assumption~\ref{ass:regularity}-(iii)-(c) and the boundedness of $dR$ on sublevel sets, we have the bound, for $0\leq t\leq t_r$,
\[
\vert v^{(r)}_t(u) \vert \leq C_1+C_2 r
\]
with constants $C_1$ and $C_2$ independent of $u,r$ and $t$. This shows, by Gr\"onwall's lemma applied to the flow of characteristics of the velocity field (this flow is defined below in Lemma \ref{lem:flowrepresentation}), that $\mu^{(r)}_t$  is concentrated on $\{ u \in \Theta\; ;\; \mathrm{dist}(u,Q_{r_0}) \leq (r_0+C_1/C_2) e^{tC_2}\}$ and thus, for all $T>0$ there exists $r>0$ such that $t_{r}>T$. Hence $\lim_{r\to \infty} t_r =\infty$ and the gradient flow from Definition~\ref{def:WassersteinGF} is uniquely well-defined on $[0,T[$ for $T>0$ arbitrary large.
\end{proof}

Let us now add a useful representation lemma for the Wasserstein gradient flow as the pushforward of $\mu_0$ by the flow of the velocity fields.

\begin{lemma}[Representation of the flow]\label{lem:flowrepresentation}
Under the assumptions of Proposition~\ref{prop:uniquenessWGF}, let $(\mu_t)_{t\geq 0}$ be the Wasserstein gradient flow of $F$ and $(v_t)_t$ the associated velocity fields. Consider the flow $X:\R_+\times \Omega\to\Omega$ which for all $u\in \Omega$, is an absolutely continuous solution to
\begin{align*}
X(0,u) = u \quad\text{and} \quad \partial_t X(t,u) = v_t(X(t,u)) \text{ for a.e.\ $t\geq 0$}.
\end{align*}
Then $X$ is uniquely well-defined, continuous, $X(t,\cdot)$ is Lipschitz on $Q_r$, uniformly on compact time intervals for all $r>0$, and it holds $\mu_t = (X_t)_\# \mu_0$.
\end{lemma}
\begin{proof}
The claims concerning $X$ are classical and follow from the fact that $v_t$ satisfies a one-sided Lispchitz property on $Q_r$, uniformly on compact time intervals~\cite[Lemma 8.1.4]{ambrosio2008gradient}. The expression as a pushforward is also a general property of the continuity equation, see~\cite[Prop. 8.1.8]{ambrosio2008gradient}.
\end{proof}

\subsection{Proof of the many-particle limit (Theorem~\ref{th:manyparticle})}
%\begin{proof}
While we could rely on abstract stability results for Wasserstein gradient flows~\cite[Thm.11.2.1 (Stability)]{ambrosio2008gradient} our proof is direct and uses basic arguments. It also gives an independent argument for the existence of Wasserstein gradient flows, distinct from the standard one : it involves a  discretization in space instead of the classical discretization in time.

\paragraph{Step (i).} We first show that, at least on a small time interval $[0,t_r]$, the paths are contained in $Q_r$ for some $r>r_0$. Let us introduce $t_r$ the first exit time from $Q_r$ 
\[
t_r := \inf \left\{ t>0\;;\; \exists m\in \N, \mu_{m,t}(Q_r)<1\right\}.
\]
In order to show that $t_r$ is strictly positive, it is sufficient to bound the velocity of individual particles before $t_r$. Consider $L_{V,r}$ the Lipschitz constant of $V$ on $Q_r$. Given the expression of the velocity of each particle (given in Eq.~\eqref{eq:velocity}) and the minimum travel distance $r-r_0$ required to exit $Q_r$, we obtain the lower bound on the exit time $t_r\geq (r-r_0)/(\Vert d\Phi \Vert_{\infty,r}\Vert dR\Vert_{\infty,r}+L_{V,r})>0$.

\paragraph{Step (ii).} Let us now work on the time interval $[0,t_r]$ and prove the existence of a limit curve $t\mapsto \mu_t$ in the space $\P_2(\Theta)$ using standard estimates for gradient flows and compactness. Our starting point is the bound, for $0\leq t_1<t_2\leq t_r$,
\[
W_2(\mu_{m,t_1},\mu_{m,t_2})^2\leq \frac1m \sum_{i=1}^m \vert \mathbf u_{m,i}(t_2)- \mathbf u_{m,i}(t_1)\vert^2 \leq  \frac{(t_2-t_1)}{m} \int_{t_1}^{t_2} \sum_{i=1}^m \vert \mathbf u'_{m,i}(s)\vert^2\d s
\]
which follows by matching each particle at $t_1$ to its future position at $t_2$, and by Jensen's inequality. Recalling the identity
$
\frac1m \sum_{i=1}^m \vert \mathbf u_{m,i}'(t)\vert^2=-\frac{d}{dt} F(\mu_{m,t})
$
from Proposition~\ref{eq:classicalGF}, it follows
\[
W_2(\mu_{m,t_1},\mu_{m,t_2}))\leq\sqrt{t_2-t_1}\left( \sup_m F(\mu_{m,0}) - \inf_{\mu\in \P(\R^d)} F(\mu))\right)^{1/2}
\]
and thus the family of curves $(t\mapsto \mu_{m,t})_m$ is equicontinuous in $W_2$ on $[0,t_r]$, uniformly in $m$. Moreover, for all $t\in [0,t_r]$, the family $(\mu_{m,t})_m$  lies in a $W_2$ ball, as such weakly precompact (but a priori not $W_2$-precompact). Since the weak topology is weaker than the topology of $W_2$, by Ascoli theorem, we can extract a subsequence converging weakly to a curve $(\mu_t)_{t\geq 0}$ continuous in the weak topology, which is concentrated in $Q_r$ at all time. We have also uniform convergence in the Bounded Lipschitz metric, which metrizes weak convergence of probability measures. In the following we only consider this subsequence, still denoted by $(\mu_m)_m$.

\paragraph{Step (iii).} The next step is to show that the limit curve $(\mu_t)$ satisfies a continuity equation as in Definition~\ref{def:WassersteinGF}. Consider the velocity fields $v_{m,t}$ defined in Equation~\eqref{eq:velocity} and let us define $v_t$ the analog for the limit curve $(\mu_t)_t$. We want to show that the sequence $(E_m)_m$ of \emph{momenta}, the vector valued measures on $[0,t_r]\times \Omega$ defined by $E_m := v_{m,t}\mu_{m,t}\d t$, converges weakly to $E:= v_t\mu_t\d t$. Notice that these measures are also concentrated on $Q_r$. For any bounded and continuous function $\varphi:[0,t_r]\times \R^d\to\R^d$, it holds
\begin{equation}\label{eq:cvgcemomentum}
\left\vert \int \varphi\cdot \d (E_m-E) \right\vert \leq
\Vert \varphi\Vert_\infty \int \vert v_{m,t}(u)-v_t(u)\vert\d \mu_{m,t}(u)\d t  
+ \left\vert  \int \varphi\cdot v_t\d (\mu_{m,t} -\mu_t) \d t\right\vert.
\end{equation}
We first prove that the first term in~\eqref{eq:cvgcemomentum} tends to $0$. Since all $(\mu_{m,t})_{m,t}$ are concentrated on $Q_r$, it is sufficient to show that the sequence of velocity fields $(t,u)\mapsto v_{m,t}(u)$ converges uniformly on $[0,t_r]\times Q_r$ to $(t,u)\mapsto v_t(u)$.
We have, using the fact that a projection on a convex set is $1$-Lipschitz, 
\[
\vert v_{m,t}(u) - v_{t}(u)\vert \leq 2\vert \tilde v_{t,m}(u)-\tilde v_t(u)\vert \leq 2 \Vert d\Phi\Vert_{\infty,r} \cdot \Vert R'(\tint \Phi\d\mu_{m,t})-R'(\tint \Phi\d\mu_{t})\Vert.
\]
Moreover, we have for all $t\in{[0,t_r]}$,
\begin{align*}
\Vert R'(\tint \Phi\d\mu_{m,t})-R'(\tint \Phi\d\mu_{t})\Vert &\leq \Vert dR\Vert_{\infty,r}\cdot \Vert \tint \Phi\d\mu_{m,t} - \tint \Phi\d\mu_{t}\Vert \\
&\leq \Vert dR\Vert_{\infty,r}\cdot  \sup_{\substack{f\in \F,\, \Vert f\Vert \leq 1}} \tint \langle f, \Phi(u)\rangle \d (\mu_{m,t}-\mu_{t})(u)  \\
&\leq\Vert dR\Vert_{\infty,r}\cdot \max\{ \Vert \Phi \Vert_{\infty,r},\Vert d\Phi\Vert_{\infty,r}\}\cdot \Vert \mu_{m,t}-\mu_t\Vert_\BL .
\end{align*}

Since the convergence of $(t\to\mu_{m,t})_m$ is uniform in the Bounded Lipschitz norm, this proves uniform convergence of the velocity fields and the convergence of the first term in~\eqref{eq:cvgcemomentum} to $0$.
The second term also converges to 
$0$ because $(t,u)\mapsto\varphi(t,u) \cdot v_t(u)$ is continuous and bounded. We thus conclude that $E_m$ tends weakly to $E$ and, in particular, the continuity equation~\eqref{eq:gradientflow} is also satisfied in the limit. As $(v_t)_t$ is bounded on $Q_r$, uniformly in time, one has $\int_0^{t_r}\int_\Omega \vert v_t(u)\vert^2 \d \mu_t(u) \d t<\infty$ which proves that $(\mu_t)_t$ is absolutely continuous in $W_2$.

\paragraph{Step (iv).}So far, we have shown the convergence, up to a subsequence, to a Wasserstein gradient flow on $[0,t_r]$: it remains to show that $\lim_{r\to \infty} t_r =\infty$. Since $F(\mu_{m,0})\to F(\mu_0)$ and all paths $(\mu_{t,m})_t$ decrease monotonically the value of $F$, everything lies in a sublevel of $R$, where $dR$ is bounded. It follows that a uniform bound on the velocity of the particles with linear growth in $r$ is available and, by Gr\"onwall's inequality, we obtain that $\lim_{r\to \infty} t_r =\infty$, just as in the end of the proof of Proposition~\ref{prop:uniquenessWGF}. The theorem follows by combining this result with the uniqueness stated in Proposition~\ref{prop:uniquenessWGF}.
%\end{proof}

% !TEX root = ../manyparticle.tex

\section{Convergence to global minimizers}\label{app:global}
We give in this section a proof of Theorems~\ref{th:mainhomogeneous} and \ref{th:mainbounded}. All results have two  versions: one in the $2$-homogeneous setting (Assumptions~\ref{ass:2homogeneous}) and its counterpart in the partially $1$-homogeneous setting (Assumptions~\ref{ass:bounded}). We have displayed in Figure~\ref{fig:homogeneous} the level sets of functions with these homogeneity properties, in order to highlight the differences between these two cases. The proofs tend to be more straightforward in the $2$-homogeneous setting and they can be read independently of the other case. This section is organized as follows:
\begin{itemize}
\item In Section~\ref{app:optimalityconditions}, we justify the global optimality conditions.
\item We give in Section~\ref{app:escape} a criteria for Wasserstein gradient flows to escape from neighborhoods of non-optimal stationary points, and we also characterize measures that can be limits of Wasserstein gradient flows. These results are valid for arbitrary initializations.
\item In Section~\ref{app:separation}, we prove that the assumption on the support of the initialization made in Theorems~\ref{th:mainhomogeneous} and \ref{th:mainbounded} is preserved by Wasserstein gradient flows.
\item All these facts combined lead to a proof of Theorems~\ref{th:mainhomogeneous} and \ref{th:mainbounded} in Section~\ref{app:mainproof}.
\end{itemize}

It will be often the case in the statements and in the proofs that they involve the projection $h^i(\mu)$ of a probability measure $\mu \in \P(\Omega)$ (with $i=1,2$) (introduced in Section~\ref{app:lifting}) instead of $\mu$ itself. This is motivated by two facts: (i) this projected measure it generally the object of interest in the optimization problem as it clears the redundancy caused by homogeneity and (ii) the assumptions that the projection $h^i(\mu_t)$ of a Wasserstein gradient flow converges is more reasonable than the convergence in $W_2$ of the original gradient flow, where generally no compactness is available.

\subsection{Optimality conditions (proof of Proposition~\ref{prop:stationnary})}\label{app:optimalityconditions}
Let us first remark that, by a first order Taylor expansion of $R$, we have that for all $\mu, \sigma \in \M(\Omega)$ with $F(\mu), F(\sigma) < \infty$, it holds $\int \vert F'(\mu) \vert d \sigma <\infty$ and
\[
\frac{d}{d\epsilon} F(\mu+\epsilon \sigma)\vert_{\epsilon=0} = \int_\Omega F'(\mu) d \sigma \quad\text{with}\quad F'(\mu): u \mapsto \langle R'(\tint \Phi \d \mu),\Phi(u)\rangle + V(u).
\]
Let $\mu,\nu \in \M_+(\Omega)$ be such that $F(\nu), F(\mu)<\infty$, consider $\sigma := \nu - \mu$ and its Lebesgue decomposition $\sigma = f\mu + \sigma^{\perp}$ where $f\in L^1(\mu)$, $\delta^{\perp}\in \M_+(\Omega)$ is singular to $\mu$ (see~\cite[Thm. 4.3.2]{cohn1980measure}). Clearly, by the above first order formula, it is necessary to have $F'(\mu)\geq 0$ everywhere with equality $\mu$-a.e., for $\mu$ to be a minimizer. It is also sufficient since in this case we have, by convexity,
\[
0 = \frac{d}{dt} F(\mu+t \sigma)\vert_{t=0} \leq   \frac{d}{dt} \left( (1-t)F(\mu)+t F(\nu)\right)\vert_{t=0} = F(\nu)-F(\mu).
\]

\subsection{A criteria to escape from non-optimal stationary points}\label{app:escape}
We now give a criteria for Wasserstein gradient flows to escape from non-optimal stationary points. It is valid both in the finite-particle regime and in the many-particle limit. Such a result supports the idea that, even in the finite-particle case (i.e.\ classical gradient flows), the point of view using measures is natural.

\begin{figure}
\centering
\begin{subfigure}{0.48\textwidth}
\centering
\includegraphics[scale=0.35]{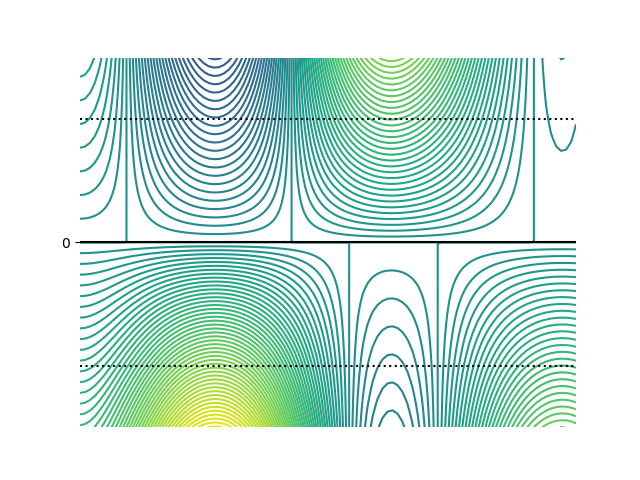}
\caption{Positively $1$-homogeneous in the vertical variable.}
\end{subfigure}%
\begin{subfigure}{0.48\textwidth}
\centering
\includegraphics[scale=0.35,trim=-1cm 0cm 1cm 0cm,clip]{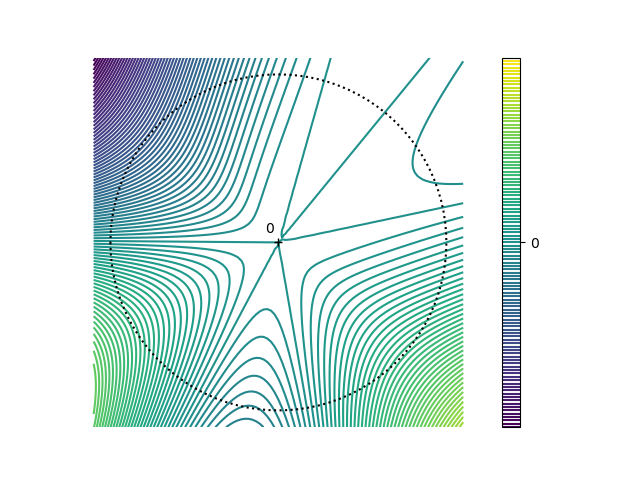}
\caption{Positively $2$-homogeneous.}
\end{subfigure}
\caption{Level sets of some functions on $\Omega = \R^2$ with homogeneity properties. The derivative $F'(\mu)$ of $F$, seen as a continuous function on $\Omega$ ``looks like'' plot (a) in the partially $1$-homogeneous case and like plot (b) in the $2$-homogeneous case. Wasserstein gradient flows of $F$ are simply set of particles $\mu_t$ that ``slide down'' such landscapes following the direction $-\nabla F'(\mu_t)$ (the subtlety being that the landscape itself depends on $\mu_t$). Minimizers $\mu^*$ of $F$ over $\M_+(\Omega)$ are characterized by the fact that this function $F'(\mu^*)$ is nonnegative on $\Omega$ and vanishes on the support of $\mu^*$. By homogeneity, it is sufficient to study these functions restricted to a subspace (dotted lines) as we do in the proofs.}
\label{fig:homogeneous}
\end{figure}

\subsubsection{The $2$-homogeneous case}
We start with the positively $2$-homogeneous setting which is slightly simpler. We consider the operator $h^2: \M_+(\R^d)\to \M_+(\S^{d-1})$ defined in Appendix \ref{app:lifting}. To simplify notations, measures on the sphere $\S^{d-1}$ are interpreted hereafter as measures on $\R^d$ concentrated on the sphere.

\begin{proposition}[Criteria to espace local minima]\label{prop:escape2}
Under Assumptions~\ref{ass:2homogeneous}, let $\mu \in \M_+(\R^d)$ be such that  $F'(\mu)$ is not nonnegative. There exists $\epsilon>0$ and a set $A\subset \Omega$ such that if $(\mu_t)_t$ is a Wasserstein gradient flow of $F$ satisfying $\Vert h^2(\mu) - h^2(\mu_{t_0})\Vert_\BL< \epsilon$ for some $t_0\geq 0$ and $\mu_{t_0}(A)>0$ then there exists $t_1>t_0$ such that $\Vert h^2(\mu) - h^2(\mu_{t_1})\Vert_\BL\geq \epsilon$. 

Such a set is given by $A=\{ r\theta \;;\; r \in {]0,\infty[}\text{ and } \theta \in K\}$ where $K$ is the $(-\eta)$-sublevel set of the restriction of $F'(\mu)$ to the unit sphere, for some $\eta>0$ that can be chosen arbitrarily close to $0$.
\end{proposition}

\begin{proof}
Let $g_\mu: \S^{d-1}\to \R$ be the restriction of $F'(\mu)$ to the unit sphere, and first assume that $0$ is in the range of $g_\mu$. Let $-\eta<0$ be a negative regular value of $g_\mu$, which is guaranteed to exist (arbitrarily close to $0$) thanks to Assumption~\ref{ass:2homogeneous} and let $K\subset \S^{d-1}$ be the corresponding sublevel set. By the regular value theorem, its boundary $\partial K = g_\mu^{-1}(-\eta)$ is a differentiable orientable compact submanifold (of the sphere) of dimension $d-2$ and is orthogonal to the gradient field of $g_\mu$. Also, $V$ is differentiable on a neighborhood of $\partial K$, by the regular value property. It holds $g_\mu(\theta)<-\eta$ for $\theta \in K$ and $\nabla g_\mu(\theta) \cdot \vec n_\theta < -\beta$ for all $\theta \in \partial K$, where $\vec n_\theta$ is the unit normal vector to $\partial K$ pointing outwards, for some $\beta>0$.  In the following lemma, we show that these properties of $K$ are also satisfied in a neighborhood of $\mu$. We denote by $\Vert \cdot \Vert_{C^1}$ the maximum of the supremum norm of a function and the supremum norm of its gradient.

\begin{lemma}\label{lem:estimates2}
Let $\phi$ be the restriction of $\Phi$ to the sphere and let $\tilde g_\mu: \theta \in \S^{d-1}\mapsto \langle R'(\int \Phi \d \mu),\phi(\theta)\rangle$. For all $C_0>0$, there exists $\alpha>0$ such that for all $\mu, \nu \in \M(\Omega)$, such that $\Vert h^2(\mu)\Vert_\BL,\Vert h^2(\mu)\Vert_\BL< \eta$ it holds
\[
\Vert \tilde g_\mu - \tilde g_\nu \Vert_{C^1} \leq \alpha \Vert \phi \Vert_{C^1}^2 \cdot \Vert h^2(\mu)- h^2(\nu)\Vert_\BL.
\]
\end{lemma}
\begin{proof}
Let us introduce $\alpha>0$ the Lipschitz constant of $dR$ on the set $\{ \int \Phi \d \mu \; ; \; \mu \in \P_2(\R^d)\;;\; h^2(\mu)<C_0 \}$ which is bounded in $\F$. It holds
\begin{align*}
\Vert \tilde g_\mu - \tilde g_\nu \Vert_{C^1} & \leq \alpha \Vert \phi \Vert_{C^1}\cdot \Vert \tint \Phi\d\mu - \tint \Phi\d\nu\Vert \\
& \leq \alpha \Vert \phi \Vert_{C^1}\cdot \Vert \tint \phi\d h^2(\mu) - \tint \phi \d h^2(\nu)\Vert \\
& \leq \alpha \Vert \phi \Vert_{C^1}\cdot \sup_{\substack{f \in \F, \Vert f\Vert \leq 1}} \tint \langle f, \phi\rangle \d (h^2(\mu) - h^2(\nu)) \\
& \leq \alpha \Vert \phi \Vert_{C^1}^2 \cdot \Vert h^2(\mu)- h^2(\nu)\Vert_\BL.
\end{align*}
where the last bound is due to the fact that $u \mapsto \langle f,\phi(u)\rangle$ is $\Vert \phi \Vert_{C^1}^2$-Lipschitz and upper bounded in norm by $\Vert \phi \Vert_{C^1}$ whenever $f\in \F$ satisfies $\Vert f \Vert \leq 1$, and can be extended from the sphere $\S^{d-1}$ to $\R^d$ as a Lipschitz function with the same constant.
\end{proof}

We now fix a large enough $C_0>0$ and consider measures $\nu$ such that $\Vert h^2(\nu)\Vert_\BL< C_0$. By posing $\epsilon = \min\{ \eta,\beta\}/(4\alpha M^2)$ where $\alpha>0$ is given by the previous lemma,  if $\Vert h^2(\nu)-h^2(\mu)\Vert < \epsilon$, then $g_\nu$ is upper bounded by $-\eta/2$ on $K$ and $\nabla g_\nu(\theta) \cdot \vec n_\theta < -\beta/2$ for all $\theta \in \partial K$. Now let us consider a Wasserstein gradient flow $(\mu_t)_t$ of $F$ such that $\mu_0$ is concentrated on $B(0,r_0)$ for some $r_0>0$ and, posing $\epsilon$ as above, we assume that $\Vert h^2(\mu_{0})-h^2(\mu)\Vert_\BL<\epsilon$. As long as this holds, the condition $\Vert h^2(\mu_t)\Vert_\BL< C_0$ for Lemma~\ref{lem:estimates2} to apply also holds (if $C_0$ was chosen large enough in the first place). Let $t_1>0$ be the first time such that $\Vert h^2(\mu_{t_1})-h^2(\mu)\Vert_\BL \geq \epsilon$, which might a priori be infinite.

Consider the flow $X$ of Lemma~\ref{lem:flowrepresentation}. By construction of the set $K$, any path $t\mapsto u_t = X(t,u_0)$ with $u_0 \in \R_+ K$  remains in $\R_+ K$ for $t\leq t_1$. Moreover, by positive $2$-homogeneity of $F'(\mu_t)$, the radial component of the velocity field is lower bounded by $r\cdot \eta$ so $\vert u_t\vert \geq \vert u_0\vert \exp(\eta t)$. In particular, for $0\leq t<t_1$ and $\xi>0$,
\[
h^2(\mu_t)(K) \geq (\xi \exp(\eta t))^2\cdot \mu_0({]\xi,\infty[}\times K).
\]
It follows that as long as, for some $\xi>0$, $\mu_0({]\xi,\infty[}\times K)>0$, then $h^2(\mu_t)(K)$ grows exponentially fast: this implies that $\mu_t$ leaves the $\Vert \cdot \Vert_\BL$ ball (notice that all measures here are nonnegative). Hence $t_1$ is finite. Finally, if we had not assumed that $0$ is in the range of $g_\nu$ in the first place, then we could simply take $K=\S^{d-1}$ and, by similar arguments, find that $\vert h^2(\mu_t)\vert(\S^{d-1})$ grows exponentially fast for $t<t_1$ if $\mu_0(\R^d \setminus \{0\})>0$.
\end{proof}

We now give a general property of the stationary points.
\begin{lemma}\label{lem:cvgcev2}
Under Assumptions~\ref{ass:2homogeneous}, let $(\mu_t)_t$ be a Wasserstein gradient flow of $F$.
If $h^2(\mu_t)$ converges weakly to $\nu \in \M_+(\S^{d-1})$, then $F'(\nu)$ vanishes $\nu$-a.e.
\end{lemma}
\begin{proof}
We again consider the function $\tilde g_\mu: \theta \in \S^{d-1}\mapsto \langle R'(\int \Phi \d \mu),\phi(\theta)\rangle$ for any measure $\mu\in \M_+(\R^d)$. At a point $\theta \in \S^{d-1}$, the velocity field $(v_t)_t$ associated to the gradient flow is obtained by applying the $2$-Lipschitz map $(\id - \proj_{\partial V(\theta)})$ to the vector of radial component $2 \tilde g_{\mu_t}(\theta)$ and tangential component $\nabla \tilde g_{\mu_t}(\theta)$. It follows, by Lemma~\ref{lem:estimates2}, that $v_{\mu_t}$ converges uniformly to $v_\nu$ on the sphere and as a consequence, $\int \vert v_t(u)\vert^2 d h^2(\mu_t)(u) \to \int \vert v_\nu(u)\vert^2 \d h^2(\nu)(u)$.
By recalling the energy identity for gradient flows $F(\mu_{s_1})-F(\mu_{s_2}) = \int_{s_1}^{s_2} \vert v_t(u)\vert^2 \d \mu_t \d t$ (see~\cite[Eq. (11.2.4)]{ambrosio2008gradient}) and this last term also equals $\int_{s_1}^{s_2} \vert v_t(u)\vert^2 \d h^2(\mu_t) \d t$ because the velocity field is positively $1$-homogeneous. This necessarily implies, since $F$ is lower bounded on $\P_2(\Omega)$, that $\vert v_\nu(u)\vert=0$ for $\nu$-a.e. $u\in \R^d$. In particular, looking at the radial component which is $2 F'(\nu) $, this implies that $F'(\nu)$ vanishes $\nu$-a.e.
\end{proof}

\subsubsection{The partially $1$-homogeneous case}
For the partially $1$-homogeneous case, we consider the operator $h^1: \M_+(\Omega)\to \M(\Theta)$ defined in Appendix \ref{app:lifting}. 

\begin{proposition}[Criteria to espace local minima]\label{prop:escape1}
Under Assumptions~\ref{ass:bounded}, let $\mu \in \M(\Omega)$ be such that  $F'(\mu)$ is not nonnegative. Then there exists $\epsilon>0$ and a set $A\subset \Omega$ such that if $(\mu_t)_t$ is a Wasserstein gradient flow of $F$ satisfying $\Vert h^1(\mu) - h^1(\mu_{t_0})\Vert_\BL< \epsilon$ for some $t_0\geq 0$ and $\mu_{t_0}(A)>0$ then there exists $t_1>t_0$ such that $\Vert h^1(\mu) - h^1(\mu_{t_1})\Vert_\BL\geq\epsilon$. 

Such a set is given by $A=(\R_+\times K^+)  \cup (\R_-\times K^-)$ where $K^+$ (respectively $K^-$) is the $(-\eta)$-sublevel set of $\theta \mapsto  F'(\mu)(1,\theta)$ (respectively of $\theta \mapsto  F'(\mu)(-1,\theta)$) for some $\eta>0$ that can be chosen arbitrarily close to $0$.
\end{proposition}

\begin{proof}
Let us suppose that $F'(\mu)$ takes a negative value on $\R_+\times \Theta$ (the case where it takes its negative values only on $\R_-\times \Theta$ is similar) and let us  introduce $g_\mu: \Theta\to \R$ the restriction of $F'(\mu)$ to $\{1\}\times \Theta$, that is $g_\mu(\theta) = \langle R'(\tint \Phi \d \mu), \phi(\theta)\rangle + \tilde V(\theta)$.
Let $-\eta<0$ be a negative regular value of $g$, which is guaranteed to exist (arbitrarily close to $0$) thanks to Assumption~\ref{ass:bounded}, and let $K^+\subset \Theta$ be the corresponding sublevel set. By the regular value theorem, its boundary $\partial K^+ = g_\mu^{-1}(-\eta)$ is a differentiable orientable manifold of dimension $d-2$ and is orthogonal to the gradient field of $g_\mu$. In the case where $\Theta$ is bounded, $\partial K^+$ is compact and, as a consequence, there is $\beta>0$ such that $\inf_{\theta\in \partial K} \vert dg_\mu(\theta)\vert\geq \beta$. If $\Theta=\R^{d-1}$ and the sublevel set $K^+$ is unbounded, then we have to choose $\eta$ so that it is also a regular value of the function on the sphere $\S^{d-2}$ to which $g_\mu$ converges uniformly at infinity. Then, the same positive lower bound holds for some $\beta>0$. It follows that on $K$, $g_\mu\leq -\eta$ and, $\nabla g_\mu(\theta) \cdot \vec n_\theta < -\beta$ for all $\theta \in \partial K$, where $\vec n_\theta$ is the unit normal vector to $\partial K$ pointing outwards. In the following lemma, we show that these properties of $K$ are also true with respect to $g_\nu$ if $\nu$ is close enough to $\mu$. We denote by $\Vert \cdot \Vert_{C^1}$ the maximum of the supremum norm of a function and the supremum norm of its gradient.

\begin{lemma}\label{lem:estimates1}
For all $C_0>0$, there exists $\alpha>0$ such that for all $\mu, \nu \in \M_+(\Omega)$ that satisfy $\Vert h^1(\mu)\Vert_\BL, \Vert h^1(\nu)\Vert_\BL< C_0$, it holds
\[
\Vert g_\mu - g_\nu \Vert_{C^1} \leq \alpha \Vert \phi \Vert_{C^1}^2 \cdot \Vert h^1(\mu)- h^1(\nu)\Vert_\BL.
\]
\end{lemma}
\begin{proof}
Let us introduce $\alpha>0$ the Lipschitz constant of $dR$ on the set $\{ \int \Phi \d \mu \; ; \; \mu \in \P(\R^d)\;;\; h^1(\mu)<C_0 \}$ which is bounded in $\F$. It holds
\begin{align*}
\Vert g_\mu - g_\nu \Vert_{C^1} & \leq \alpha \Vert \phi \Vert_{C^1}\cdot \Vert \tint \Phi\d\mu - \tint \Phi\d\nu\Vert \\
& \leq \alpha \Vert \phi \Vert_{C^1}\cdot \Vert \tint \phi\d h^1(\mu) - \tint \phi \d h^1(\nu)\Vert \\
& \leq \alpha \Vert \phi \Vert_{C^1}\cdot\sup_{\substack{f \in \F, \Vert f\Vert \leq 1}} \tint \langle f, \phi\rangle \d (h^1(\mu) - h^1(\nu)) \\
& \leq \alpha \Vert \phi \Vert_{C^1}^2 \cdot \Vert h^1(\mu)- h^1(\nu)\Vert_\BL.
\end{align*}
where the last bound is due to the fact that $u \mapsto \langle f,\phi(u)\rangle$ is $\Vert \phi \Vert_{C^1}$-Lipschitz and upper bounded in norm by $\Vert \phi \Vert_{C^1}$ whenever $f\in \F$ satisfies $\Vert f \Vert \leq 1$.
\end{proof}

We now fix a large enough $C_0>0$ and consider measures $\nu$ such that $\Vert h^1(\nu)\Vert_\BL< C_0$. By posing $\epsilon = \min\{ \eta,\beta\}/(4\alpha M^2)$ where $\alpha>0$ is given by the previous lemma,  if $\Vert h^1(\nu)-h^1(\mu)\Vert < \epsilon$, then $g_\nu$ is upper bounded by $-\eta/2$ on $K$ and $\nabla g_\nu(\theta) \cdot \vec n_\theta < -\beta/2$ for all $\theta \in \partial K$. Now, let us consider a Wasserstein gradient flow $(\mu_t)_t$ of $F$ such that $\mu_0$ is concentrated on $[-r_0,r_0]\times \Theta$ for some $r_0>0$ and $\Vert h^1(\mu_{0})-h^1(\mu)\Vert_\BL<\epsilon$. As long as this holds, the condition $\Vert h^1(\mu_t)\Vert_\BL< C_0$ for Lemma~\ref{lem:estimates1} to apply also holds. Let $t_1>0$ be the first time such that this last condition is not satisfied, which might a priori be infinite.

Consider the flow $X$ of Lemma~\ref{lem:flowrepresentation}. By construction of the set $K^+$, any path $t\mapsto (w_t,\theta_t) = X(t,(w_{0},\theta_{0}))$ with $(w_0,\theta_0)\in \R_+\times K$  remains in $\R_+\times K^+$ for $t\leq t_1$. Moreover, by homogeneity of $F'(\mu_t)$ in the variable $w$, the component of the velocity field on $w$ is lower bounded by $\eta/2$ so $ w_t \geq w_0 + t\cdot \eta/2$. For similar reasons, no path enters the set $\R_-\times K^+$ during this time interval and the paths inside this set satisfy $w_t \geq w_0 +t\cdot \eta/2$ (this follows by the fact that $F'(-1,\cdot) \geq F'(1,\cdot)$). In particular, for $0\leq t<t_1$,
\[
h^1(\mu_t)(K^+)\geq  (t\cdot \eta/2)\cdot \mu_0(\R_+\times K^+)  + \min \{ 0, t\cdot \eta/2 - r_0\}\cdot \mu_0(\R_-\times K^+).
\]
So we see that as long as $\mu_0(\R_+\times K)>0$, then $h^1(\mu_t)(K^+)$ grows at least linearly. 

If $\Theta = K^+$ then the previous lower bound immediately implies that $t_1$ is finite (choose the constant unit function in the definition of the norm $\Vert \cdot \Vert_\BL$). Otherwise, in order to finalize our proof, we need to make sure that the mass $h^1(\mu_t)(K^+)$ does not grow unbounded just near the boundary of $K^+$. To do so, let us consider another sublevel set $\tilde K^+$ of $g_\mu$ associated to another regular value in the range $\tilde \eta \in ]-\eta,0[$ and such that $\tilde K^+$ does not cover $\Theta$. As $g_\mu$ is Lipschitz, there exists $\Delta\in]0,1]$ such that the distance between $K^+$ and $\Theta \setminus \tilde K^+$ is at least $\Delta$. Taking another, smaller radius $\epsilon>0$ if necessary, by similar arguments as above, either $t_1$ is smaller than $2r_0/\tilde \eta$, or there exists $\tilde t>t_0$ such that $h^1(\mu_t)$ is nonnegative on $\tilde K^+$ for $t\in [\tilde t,t_1[$. Choosing, as a test function in the definition of the norm $\Vert \cdot \Vert_\BL$, the distance to the set $\Theta \setminus \tilde K$ clipped to $1$, one obtains, for $t \in [\tilde t, t_1[$,
\[
\Vert h^1(\mu_t)\Vert_\BL \geq \Delta \cdot h^1(\mu_t)(K^+)
\]
which also grows at least linearly with $t$. So $h^1(\mu_t)$ eventually leaves any $\Vert \cdot \Vert_\BL$-ball, hence $t_1$ is finite.
\end{proof}

As for the $2$-homogeneous case, we give a general property of the stationary points.
\begin{lemma}\label{lem:cvgcev1}
Under Assumptions~\ref{ass:bounded}, let $(\mu_t)_t$ be a Wasserstein gradient flow of $F$.
If $h^1(\mu_t)$ converges weakly to $\nu \in \M_+(\Theta)$, then $F'(\nu)$ vanishes $\nu$-a.e.
\end{lemma}
\begin{proof}
At a point $(1,\theta)\in \Omega$, with $\theta \in \Theta$, the velocity field $(v_t)_t$ associated to the gradient flow is given by applying the $2$-Lipschitz map $(\id - \proj_{\partial V(1,\theta)})$ to the vector with first component $g_{\mu_t}(\theta)$ and other components $\nabla g_{\mu_t}(\theta)$, where $g_{\mu_t}$ is defined as in the proof of Proposition~\ref{prop:escape1}. It follows, by Lemma~\ref{lem:estimates2}, that $v_{\mu_t}$ converges uniformly to $v_\nu$ on $\{1\} \times \Theta$. However, in contrast to Lemma~\ref{lem:cvgcev1}, the energy dissipation identity is not invariant by the projection operator, so we have to develop arguments similar to those used to prove Proposition~\ref{prop:escape1} (we do so with less details). Using the uniform convergence of $g_{\mu_t}$ if there exists $\theta_0\in \Theta$ such that $g_\nu(\theta)>0$ then we can build a set $\R^+ \times K$ with $\theta_0 \in \mathrm{int} K$  such that for some $t_0>0$, no trajectory of the flow $X_t$ enters this set after $t_0$ and the component of the velocity on $w$ is upper bounded by $-g_\nu(\theta)/2$. Since $\mu_{t_0}$ is concentrated on a set $Q_{r_0}$, this implies that $\mu_{t_0}(\R^*_+\times K)$ vanishes in finite time and in particular, $\nu(K)=0$. Thus we have shown that $F'(\nu)$ is nonpositive $\nu$-a.e. Also it can be deduced by Proposition~\ref{prop:escape1}, that $F'(\nu)$ is nonnegative $\nu$-a.e. So $F'(\nu)$ vanishes $\nu$-a.e.
\end{proof}

\subsection{Stability of separation properties}\label{app:separation}
Here we prove the fact that the separation properties of the support used in Theorems~\ref{th:mainbounded} and \ref{th:mainhomogeneous} are preserved by Wasserstein gradient flows. We give a proof based on \emph{topological degree} theory: this tool allows to cover the case of discontinuous velocity fields, which appear when $V$ is non-differentiable. In a more regular setting, the facts that follow are easier to prove because then, $\mu_t$ is the pushforward of $\mu_0$ by a homeomorphism. Let us give a definition of the topological degree sufficient to our setting.

\begin{definition}[Topological degree]\label{def:degree}
Let $f:\R^d\to \R^d$ be a continuous map, $A\subset \R^d$ a bounded open set and $y\notin f(\partial A)$. The topological degree $\deg(f,A,y)$ is a signed integer that satisfies:
\begin{enumerate}
\item If $\deg(f,A,y)\neq 0$ then there exists $x\in A$ such that $f(x)=y$. If $y\in A$ then $\deg(\id,A,y)=1$.
\item If $A_1,A_2$ are disjoint open subsets of $A$ and $y \notin f(\overline A \setminus (A_1\cup A_2))$ then $\deg(f,A,y)=\deg(f,A_1,y)+\deg(f,A_2,y)$.
\item
If $X:[0,1]\times \R^d \to \R^d$ is continuous and $y:[0,1]\to \R^d$ is a continuous curve such that $y(t)\notin X_t(\partial A)$ for all $t\in [0,1]$, then $\deg(X_t,A,y_t)$ is constant on $[0,1]$.
\end{enumerate}
\end{definition}
These properties characterize a uniquely well-defined map $\deg$ from the set of triplets $(f,A,y)$ as above to the set of signed integers~\cite[Thm. 1-2]{browder1983fixed}. Intuitively, it gives an \emph{algebraic} count of the number of solutions to $f(x)=y$ for $x\in A$, where algebraic means that a solution $x$ counts as $+1$ if $f$ preserves orientation around $x$ and $-1$ otherwise.

The following lemma shows that taking the support of a measure and its pushforward by a continuous map are operations that almost commute. They commute for instance if the map is \emph{closed} (i.e.\ maps closed sets to closed set).
\begin{lemma}\label{lem:supportpushforward}
If $f:\R^d\to \R^d$ is a continuous map and $\mu\in \M_+(\R^d)$, then $\spt(f_\# \mu)=\overline{f(\spt \mu)}$. 
\end{lemma}
\begin{proof}
Let $y\in f(\spt \mu)$ and $\V$ a neighborhood of $y$. By continuity, $f^{-1}(\V)$ is the neighborhood of a point in $\spt \mu$ so $0< \mu(f^{-1}(\V))=f_\#\mu (\V)$, hence $y \in \spt f_\# \mu$ so $f(\spt \mu) \subset \spt f_\# \mu$. Conversely, let $y\in \overline{f(\spt \mu)}^c$ and let $\V$ a neighborhood of $y$ that does not intersect $\overline{f(\spt \mu)}$. This neighborhood satisfies $f^{-1}(\V)\subset (\spt \mu)^c$, so it holds $f_\# \mu (\V) = \mu(f^{-1}(\V))\leq \mu ((\spt \mu)^c)=0$. Hence $y \in (\spt f_\#\mu)^c$ so $\overline{f(\spt \mu)}^c \subset (\spt f_\#\mu)^c$ which implies $\spt f_\#\mu \subset \overline{f(\spt \mu)}$.
\end{proof}

\subsubsection{The $2$-homogeneous case}
We first state the property and the stability result that we wish to establish in the $2$-homogeneous setting.

\begin{property}[Separation, $2$-homogeneous case]\label{ass:separationhomogeneous}
$K$ is a closed subset of $\R^d$ contained in $B(0,r_b)$ that separates $r_a\S^{d-1}$ from $r_b\S^{d-1}$, for some $0<r_a<r_b$.
\end{property}

\begin{lemma}[Stability of the separation property]\label{lem:preservedhomogeneous}
Under Assumptions~\ref{ass:2homogeneous}, let $(\mu_t)_{t\geq 0}$ be a Wasserstein gradient flow of $F$. If the support of $\mu_0$ satisfies Property~\ref{ass:separationhomogeneous}, so does the support of $\mu_t$, for all $t>0$.
\end{lemma}

Note that this property is generally lost \emph{in the limit} $t\to \infty$. This lemma is a consequence of the following, more abstract proposition, that deals with sets instead of measures. The reader can keep in mind that we will apply this result with $X$ being the flow of the velocity field introduced in Lemma~\ref{lem:flowrepresentation} and $K$ being the support of $\mu_0$.

\begin{proposition}[Set separation, spheres]\label{prop:degreespherical}
Consider a continuous map $X:[0,T]\times \R^d \to \R^d$ such that $X(0,\cdot)=\id$, and such that, for all $\epsilon>0$, there exists $\eta>0$ such that $u\in B(0,\eta)$ implies $X_t^{-1}(u)\subset B(0,\epsilon)$. If $K$ satisfies Property~\ref{ass:separationhomogeneous}, then $X_t(K)$ satisfies the same property for all $t\in [0,T]$.
\end{proposition}

\begin{proof}
Let $0<\epsilon<\alpha<\beta$ be such that $X_t(K) \subset B(0,\alpha-\epsilon)$ and $ B(0,\alpha+\epsilon) \subset X_t(B(0,\beta))$ for all $t\in [0,T]$ and let $A$ be the intersection of $B(0,\beta)$ with the (unique) unbounded connected component of $\R^d\setminus K$. Consider the function $\tilde X : (t,x)\mapsto (t,X_t(x))$ and the set $S = \tilde X ([0,T]\times \partial A)$ which is a compact subset of $[0,T]\times \R^d$. Since connected components of $S^c$ (the complement of $S$ in $[0,T]\times \R^d$) are path connected, recalling Definition~\ref{def:degree}, it follows that 
\[
(t,x) \mapsto \deg(X_t,A,x) 
\] 
is constant on each connected component of $S^c$. Moreover, this degree equals $1$ if the connected component intersects $\{0\}\times A$ and $0$ if it intersects $\{0\}\times (\R^d\setminus A)$. In particular, this degree is $1$ on $[0,T]\times \alpha \S$ and, by the assumptions on $K$ and $X$, there is a small tube $[0,T]\times B(0,\eta)$ where the above degree is $0$.  
So for a fixed $t\in[0,T]$, any path joining $(\eta/2)\S^{d-1}$ to $\alpha \S^{d-1}$ must intersect $X_t(\partial A)$. We restrict our attention to paths confined in $B(0,\alpha)$ and it remains to notice that $\partial A \subset K \cup \beta \S$ and so 
\[
X_t(\partial A) \cap B(0,\alpha) \subset (X_t(K) \cup X_t(\beta \S))\cap B(0,\alpha) = X_t(K) 
\]
This shows that any path joining $(\eta/2)\S$ to $\alpha \S$ must intersect $X_t(K)$.
\end{proof}

\begin{proof}[Proof of Lemma~\ref{lem:preservedhomogeneous}]
Consider the continuous flow $X:[0,T]\times \R^d\to \R^d$  introduced in Lemma~\ref{lem:flowrepresentation}. For all $t\in [0,T]$, the map $X_t$ is coercive and thus closed, so Lemma~\ref{lem:supportpushforward} applies and gives $\spt ((X_t)_\# \mu_0) = X_t(\spt \mu_0)$. We just have to check the assumption of Proposition~\ref{prop:degreespherical} concerning the stability of the inverse map of $X_t$ near $0$. Since $\Phi$ and $V$ are $2$-homogeneous, we have that $d\Phi$ and $\partial V$ that are $1$-homogeneous and thus, there exists a constant $C>0$ such that $\vert v_t(u)\vert \leq C \vert u\vert$ for all $t \in [0,T]$ and $u\in \R^d$. This upper bound on the velocity implies in particular that if $u_0$ is at a distance $\vert u_0\vert$ from $0$, then it is at least at a distance $\vert u_0\vert \exp(-Ct)$ for $t\in [0,T]$.
\end{proof}

\subsubsection{The partially $1$-homogeneous case}
Here are the analogous separation property and stability lemma for the partially $1$-homogeneous case.
\begin{property}[Separation, partially $1$-homogeneous case]\label{ass:separationbounded}
$K$ is a closed set contained in a box $Q_r := [-r,r]\times \Theta$ and separates $\{-r\}\times \Theta$ from $\{r\}\times \Theta$ for some $r>0$ (in the ambiant space $\Omega = \R \times \Theta$).
\end{property}
\begin{lemma}[Stability of the separation property]\label{lem:preservedbounded}
Under Assumptions~\ref{ass:bounded}, let $(\mu_t)_{t\geq 0}$ be a Wasserstein gradient flow of $F$. If the support of $\mu_0$ satisfies Property~\ref{ass:separationbounded}, then so does the support of $\mu_t$, for all $t>0$.
\end{lemma}

Similarly as above, we first prove an abstract topological result.
\begin{proposition}[Set separation, boxes]\label{prop:setbounded}
Let $\Theta\subset \R^d$ be the closure of a bounded, connected, open set and, for some $T>0$, let $X:[0,T]\times (\R \times \Theta) \to \R \times \Theta$  be a continuous map such that $X(0,\cdot)=\id$ and $X_t(\R \times \partial \Theta) \subset \R\times \partial \Theta$ for all $t\in [0,T]$. If $K$ satisfies Property~\ref{ass:separationbounded}, then $X_t(K)$ satisfies Property~\ref{ass:separationbounded} for all $t\in [0,T]$.
\end{proposition}
\begin{proof}
Let $0<\epsilon<\alpha<\beta$ be such that $X_t(K) \subset {]-\alpha-\epsilon,\alpha+\epsilon[}\times \Theta$ and $[-\alpha,\alpha]\times \Theta \subset X_t( {]-\beta-\epsilon,\beta+\epsilon[}\times \Theta)$ for all $t\in [0,T]$, and let $A$ be  the intersection of $]-\beta,\beta[\times \Theta$ with the (unique) connected component of $(\R\times \Theta)\setminus K$ that contains $\{\alpha\} \times \Theta$. The set $A$ is bounded and open in $\R \times \R^{d-1}$. Consider the function $\tilde X : (t,x)\mapsto (t,X_t(x))$ and the set $S = \tilde X ([0,T]\times \partial A)$ which is a compact subset of $[0,T]\times (\R\times \Theta)$. Since connected components of $S^c$ (the complement of $S$ in $[0,T]\times (\R\times \Theta)$) are path connected, recalling Definition~\ref{def:degree}, it follows that 
\[
(t,(w,\theta)) \mapsto \deg(X_t,A,(w,\theta))
\] 
is constant on each connected component of $S^c$. Moreover, this degree is $1$ on $[0,T]\times(\{\alpha\} \times \Theta)$ and is $0$ on $[0,T]\times(\{-\alpha\} \times \Theta)$. So for a fixed $t\in [0,T]$, any path joining $\{-\alpha\} \times \Theta$ to $\{\alpha\} \times \Theta$ must intersect $X_t(\partial A)$. It is in particular true for paths entirely contained in ${[-\alpha,\alpha]}\times \mathrm{int}\, \Theta$. It remains to notice that $\partial A \subset K \cup (\R \times \partial \Theta) \cup (\{\beta\} \times \Theta)$ and so, thanks to our assumption on $X$,
\[
X_t(\partial A) \cap ({[-\alpha,\alpha]}\times \mathrm{int}\, \Theta) \subset X_t(K).
\]
This shows that $X_t(K)$ separates $\{-\alpha\} \times \Theta$ from $\{\alpha\} \times \Theta$ in $\R \times \mathrm{int}\, \Theta$ and in fact also in $\R\times\Theta$ because $X_t(K)$ is closed.
\end{proof}

\begin{proof}[Proof of Lemma~\ref{lem:preservedbounded}]
Let $X$ be the flow of the velocity fields introduced in Lemma \ref{lem:flowrepresentation}. It is continuous and satisfies $\mu_t = (X_t)_\# \mu_0$. Moreover, $X_0=\id$ and $X_t$ is closed because it is coercive.  We have to deal with two cases: when $\Theta$ is bounded and when $\Theta= \R^{d-1}$. In the first case, by Lemma~\ref{lem:supportpushforward}, it is sufficient so verify the assumptions of Proposition~\ref{prop:setbounded}. This reduces to making sure that $X_t(\R \times \partial \Omega)\subset \R \times \partial \Omega$, which is guaranteed by the Neumann boundary conditions. 
For the unbounded case, we bring ourselves back to the bounded case, by means of the diffeomorphism $\psi: \R\times \R^d \to \R \times \mathrm{int}\, B(0,1)$ defined by $\psi (w,\theta)=(w,(\theta/\vert \theta\vert)\cdot \tanh \vert \theta \vert)$ if $\theta\neq 0$ and $\psi(w,0)=(w,0)$ otherwise. Let $Y_t := \Psi \circ X_t \circ \Psi^{-1}$ be the flow where the second variable is mapped to the open unit ball. By direct calculus, one sees that $Y_t$ is the flow of the velocity field $\tilde v_t(y) = d\psi_{\psi^{-1}(y)} (v_t \circ \psi^{-1}(y))$ defined on $\R \times \mathrm{int}\, B(0,1)$ and that can be extended by continuity on $\R \times \S^{d-2}$ by $(g_\infty(\theta)\cdot \sign w,0)$ where $g_\infty$ is the limit which existence is assumed in Assumption~\ref{ass:bounded}-\eqref{subass:boundary} and, here, $\sign 0 = 0$. As $Y_t$ satisfies the properties of Proposition~\ref{prop:setbounded}, the conclusion of this proposition holds for the set $\psi (\spt \mu_t) = \psi \circ X_t (\spt \mu_0)$. Since $\psi$ is a diffeomorphism, it preserves connectedness properties and Lemma~\ref{lem:preservedbounded} is proved.
\end{proof}

\subsection{Main theorems: proofs and generalization}\label{app:mainproof}

First, let us state a lemma that relates the convergence of the Wasserstein gradient flows to an \emph{asymptotic property}  for the classical gradient flows, when $m,t \to \infty$. This result is used in the last claims of Theorems~\ref{th:mainhomogeneous} and~\ref{th:mainbounded}.

\begin{lemma}
Under Assumptions~\ref{ass:regularity}, let  $(\mu_t)$ be a Wasserstein gradient flow which initialization is concentrated on a set $Q_{r_0}$ and such that $F(\mu_t)\to F^*$. If $(\mu_{0,m})_m$ is a sequence of measures concentrated on a set $Q_{r_0}$ that converges to $\mu_0$ in $W_2$, then 
\[
F^* = \lim_{t \to \infty} \lim_{m\to \infty} F(\mu_{m,t}) = \lim_{m \to \infty} \lim_{t\to \infty} F(\mu_{m,t}). 
\]
\end{lemma}
\begin{proof}
The first double limit where $m$ goes first to $\infty$ is a consequence of Theorem~\ref{th:manyparticle} combined with the continuity of $F$ for the Wasserstein metric, proved in Lemma~\ref{lem:Fcontinuous}. The other double limit is obtained by the mononicity of $F(\mu_t)$ along Wasserstein gradient flows. Indeed, for all $\epsilon>0$, there exists $t_0\in \R_+$ such that $F(\mu_{t_0})<F^*+\epsilon/2$ and by Theorem~\ref{th:manyparticle}, there is $m_0\in \N$ such that for all $m\geq m_0$, $F(\mu_{t_0,m})<F(\mu_{t_0})+\epsilon/2$. Since $t\mapsto F(\mu_{m,t})$ is decreasing and lower bounded for all $m\in\N$ , it follows
\[
\forall m>m_0,\; \lim_{t\to \infty} F(\mu_{m,t}) \leq F(\mu_{m,t_0}) < F^*+\epsilon
\]
which proves the second limit.
\end{proof}

\subsubsection{The $2$-homogeneous case}

\begin{theorem}\label{th:cvgce2}
Under the assumptions of  Theorem~\ref{th:mainhomogeneous}, if $h^2(\mu_t)$ converges weakly, then its limit is a global minimizer of $F$ over $\M_+(\Omega)$ and $\lim_{t\to \infty} F(\mu_t)= F^*$.
\end{theorem}
This statement is stronger than Theorem~\ref{th:mainhomogeneous}: indeed, if $\mu_t$ converges for the Wasserstein metric, then $h^2(\mu_t)$ converges weakly (but the converse is generally not true).
\begin{proof}
Let $\nu \in \M_+(\S^{d-1})$ be the weak limit of $h^2(\mu_t)$. By Lemma~\ref{lem:cvgcev2}, $F'(\nu)$ vanishes $\nu$-a.e. For the sake of contradiction, assume that $\nu$ is not a minimizer of $F$ over $\M_+(\Omega)$: this implies that $F'(\nu)$ is not nonnegative. Let $A\subset \Omega$ and $B_\BL\subset \M(\S^{d-1})$ be the set and the $\Vert \cdot \Vert_\BL$-ball provided by Proposition~\ref{prop:escape2}. As $h^2(\mu_t)$ converges weakly, there exists $t_0>0$ such that for all $t>t_0$, $h^2(\mu_t)\in B_\BL$. But by Lemma~\ref{lem:preservedhomogeneous}, $\mu_{t_0}(A)>0$ and, by Proposition~\ref{prop:escape2}, there exists $t_1>t_0$ such that $\mu_{t_1} \notin B_\BL$, which is a contradiction so $\nu$ is minimizer of $F$ over $\M_+(\Omega)$. The second claim is a consequence of the continuity of $F$ (Lemma~\ref{lem:Fcontinuous}).
\end{proof}

\subsubsection{The partially $1$-homogeneous case}
Again, we prove a statement in terms of the projected measures: Theorem~\ref{th:mainbounded} can be deduced as an immediate corollary. Some highlights of the proof are given in Figure~\ref{fig:proofhighlights}.
\begin{theorem}\label{th:cvgce1}
Under the assumptions of  Theorem~\ref{th:mainbounded}, if $h^1(\mu_t)$ converges weakly, then its limit is a global minimizer of $F$ over $\M_+(\Omega)$ and $\lim_{t\to \infty} F(\mu_t)= F^*$.
\end{theorem}
\begin{proof}
Let $\nu$ be the weak limit of $h^1(\mu_t)$ and we see it as a measure on $\{1\} \times \Theta$. By Lemma~\ref{lem:cvgcev1}, $F'(\nu)$ vanishes $\nu$-a.e. For the sake of contradiction, assume that $\nu$ is not a minimizer of $F$ over $\M_+(\Omega)$: this implies that $F'(\nu)$ is not nonnegative. Let $A\subset \Omega$ and $\epsilon$ be the set and the radius of the $\Vert \cdot \Vert_\BL$-ball which are provided by Proposition~\ref{prop:escape2}. As $h^1(\mu_t)$ converges weakly, there exists $t_0>0$ such that for all $t>t_0$, $\Vert h^1(\mu_t)-\nu \Vert_\BL < \epsilon$. In the favorable case where $\mu_{t_0}(A)>0$ then we can conclude as in the $2$-homogeneous case, \emph{but this is not immediately guaranteed by Lemma~\ref{lem:preservedhomogeneous}}: the situation is thus trickier than in the proof of the $2$-homogeneous case.

We take notations from the proof of Proposition ~\ref{prop:escape1} and consider first the case when $\Theta$ is bounded. Let $\theta_0 \in K^+$ be a local minimum of $g_\nu$ in the interior of $K^+$ relatively to $\Theta$ (the case when $K^+$ is empty but $K^-$ is not could be treated similarly). Thanks to Neumann boundary conditions, it holds $\nabla g_\nu(\theta_0)=0$, even when $\theta_0$ lies on the boundary of $\Theta$. By Lemma~\ref{lem:preservedhomogeneous}, the line $\R \times \{\theta_0\}$ intersects the support of $\mu_{t_0}$. If this intersection lies in $\R_+\times K^+$, we can conclude immediately by Proposition~\ref{prop:escape2}. Otherwise, we fix $M>0$ such that $\mu_{t_0}$ is concentrated on $[-M,M]\times \Theta$ and we resort to applying Lemma~\ref{lem:rideridge} below. 

Let $r_0>0$ be such that $B(\theta_0,r_0)\cap \Theta \subset K^+$. By Lemma~\ref{lem:rideridge}, there exists $t_1>t_0$ such that if the support of $\mu_{t_1}$ intersects $[-M,0]\times \{\theta_0\}$ then it intersects $\R_+\times K^+$ at a subsequent time and again, we can conclude by Proposition~\ref{prop:escape2}. So it remains to check that the support of $\mu_{t_1}$ intersects $[-M,0]\times \{\theta_0\}$; the difficulty here is that $M$ was chosen prior to $t_1$.  The justification is as follows:  by Lemma~\ref{lem:preservedhomogeneous}, the support of $\mu_{t_1}$ intersects $\R_-\times \{\theta_0\}$ at a point $(w_0,\theta_0)$. The properties of $K^+$  imply that the pre-image by the flow $X_t$ of $(w_0,\theta_0)$ is included in $[-M,0[\times K^+$ and, since the $w$-component of the velocity field is lower bounded on $\R_-\times K^+$ by $\eta/2$ for $t>t_0$, one has $w_0>M$.

As for the case when $\Theta=\R^{d-1}$, we can reproduce the proof above by mapping the flow to the unit sphere, as done in the proof of Lemma~\ref{lem:preservedbounded}. The last claim of the theorem is a consequence of the continuity of $F$ (Lemma~\ref{lem:Fcontinuous}).
\end{proof}
In the unfavorable case encountered in the proof of Theorem~\ref{th:cvgce1}, we had to invoke the following lemma. It has a different nature than the other results of this paper because it relies on an explicit integration of the trajectories of the gradient flow, which means that it depends on the choice of the metric.
\begin{lemma}\label{lem:rideridge}
Consider, for a measure $\nu \in \M(\Theta)$, a point $\theta_0\in \Theta$ such that $\vert \nabla g_\nu(\theta)\vert =0$ and $g_\nu(\theta)\leq -\eta$ for some $\eta>0$. For any $M>0$ and $r_0>0$, there exists $T,\epsilon>0$ such that if $(\mu_t)_t$ is a Wasserstein gradient flow of $F$ that satisfies for all $t\in [0,T]$,  $\Vert g_{\mu_t}-g_\nu \Vert_{C^1} \leq \epsilon$ and denoting $(w(t),\theta(t))$ the solution of the flow of Lemma~\ref{lem:flowrepresentation} starting from $(w_0,\theta_0)$ with $w_0\in [-M,0]$, it holds $w(T)= 0$ and $\vert \theta(T)-\theta_0 \vert < r_0$.
\end{lemma}
\begin{proof}
The Lipschitz regularity of $g_\nu$ and its derivative implies that there exists $L>0$ such that $\max \{ \vert g_\nu(\theta)-g_\nu(\theta_0)\vert, \vert \nabla g_\nu(\theta)-\nabla g_\nu(\theta_0)\vert\} \leq L\vert \theta - \theta_0\vert$ for all $\theta \in \Theta$. Without loss of generality, we assume that $r_0< \eta/(4L)$. Consider $\epsilon \in {]0,\eta/4[}$ and assume that there exists $\bar T>0$ such that $\Vert g_{\mu_t}-g_\nu \Vert_{C^1} \leq \epsilon$ for $t\in [0,\bar T]$. Writing $q(t)=\vert \theta(t)-\theta_0\vert$, it holds for $t\in [0,\bar T]$,
\begin{equation*}
\left\{
\begin{aligned}
\frac{dq}{dt} &\leq -w(\epsilon +Lq)\\
\frac{dw}{dt} & \geq \eta - \epsilon -Lq
\end{aligned}
\right.
\end{equation*}
In particular, if we can make sure that $\vert q(t)\vert<\bar r$ for $t\in [0,\bar T]$ and if $\bar T>2/\eta$ then, as $(dw/dt)\geq \eta/2$ on this interval, there exists $T<2/\eta$ such that $w(T)=0$. 

It remains to make sure that we indeed have $\vert q(t)\vert<\bar r$ for $t\in [0,T]$, by adjusting if necessary the value of $\epsilon$. Parametrizing in $w$ instead of $t$ (it is an admissible reparametrization thanks to the positive lower bound on its derivative), we get 
\[
d q /d w =(d q /d t) \cdot (dt/dw)\leq -w(\epsilon +Lq)\cdot 2/\eta.
\]
We can apply Gr\"onwall's lemma to $\tilde q(w) = \epsilon + Lq(w)$ which satisfies $(d/dw)\tilde q(w)\leq (-2L/\eta)\cdot w\cdot \tilde q (w)$ and obtain 
\[
\tilde q(w) \leq \tilde q(w_0) \exp\left( -(2L/\eta)\tint_{w_0}^0 s d s  \right) = \epsilon \exp(Lw_0^2/\eta) .
\]
Thus, choosing $\epsilon< Lr_0/(\exp(Lw_0^2/\eta)-1)$, it  is guaranteed that $q(t)\leq r_0$ for $t\in [0,T]$.
\end{proof}

\begin{figure}
\centering
\includegraphics[scale=0.6,trim=-2cm 0cm 0cm 1cm,clip]{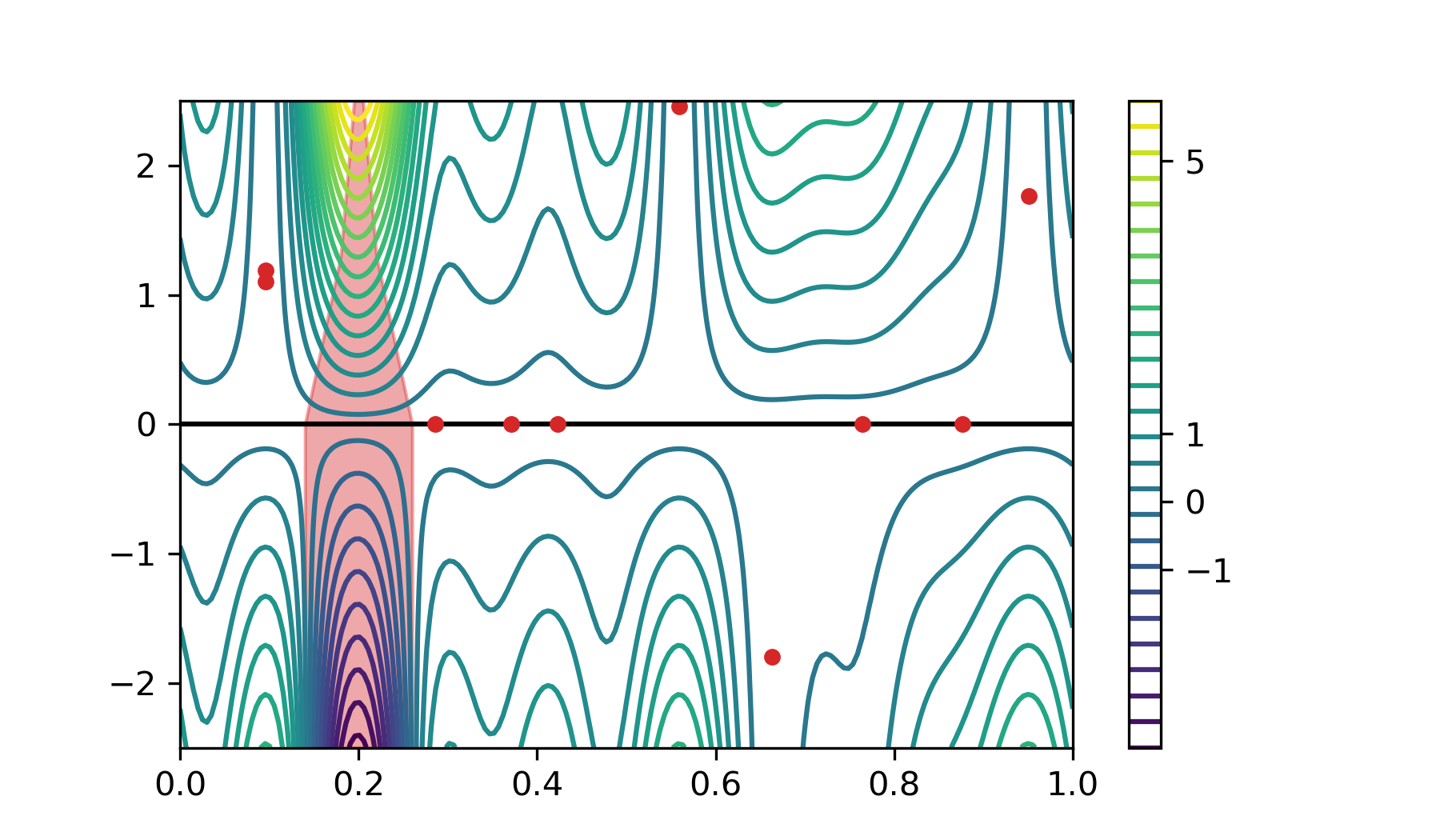}
\caption{Example, in the partially $1$-homogeneous setting, of a stationary point $\nu$ that is not a minimizer of $F$ over $\M_+(\Omega)$ since $F'(\nu)$ is not nonnegative. The variable $w$ corresponds to the vertical axis, the support of $\nu$ is the discrete set of red dots, and the lines show the level sets of $F'(\nu)$. There is a neighborhood of $\nu$ such that if a Wasserstein gradient flow $\mu_t$ enters this neighborhood and gives mass to a certain set (shown in red) then $\mu_t$ subsequently escapes this neighborhood. In Proposition~\ref{prop:escape1} this is proved for the lower part of this set (under the horizontal axis, where $F'(\nu)$ is negative). Technical Lemma~\ref{lem:rideridge} is concerned with building the upper part of this set. The proof of Theorem~\ref{th:cvgce1} lies on the fact that any measure satisfying the separation Property~\ref{ass:separationbounded} gives mass to this set.} \label{fig:proofhighlights}
\end{figure}

\subsection{Remarks}\label{app:GFconvergence}
We conclude this theoretical section with two opening remarks related to the global convergence theorems.
\paragraph{Convergence of the gradient flow.}
In the statements of Theorems~\ref{th:mainhomogeneous} and~\ref{th:mainbounded}, the convergence of the Wasserstein gradient flow comes as an assumption. In order to prove convergence of gradient flows, one generally needs two properties: (i) compactness of the trajectories and (ii) a so-called \emph{\L{}ojasiewicz inequality} which, intuitively, controls how much a function flattens around its critical points. As compactness in $W_2$ is a very strong requirement, we have relaxed the topology where convergence is required to obtain more reasonable assumptions. Yet, even when a gradient flow lies in a compact set, there are some cases where it does not converge. There has been recent progress on related issues with the study of \L{}ojasiewicz inequalities in Wasserstein space~\cite{blanchet2016family,hauer2017kurdyka}, but to our knowledge, no general result is known in our non-geodesically convex case. 

\paragraph{Towards quantitative statements.}
We stress that Propositions~\ref{prop:escape2} and~\ref{prop:escape1} provide with an intuitive criterion for a particle gradient flow to escape local minimum: roughly, it is sufficient that, when it passes close to a local minimum, at least one particle belongs to a $0$-sublevel set of the current potential $F'(\mu)$. In this paper we exploit this property by studying the many-particle limit, but other approaches are worth exploring. For instance, we could estimate the size of this sublevel set in specific cases, and use it as an indication for the particle-complexity to attain global minimizers. A discussion on a specific example is given in Section~\ref{app:details}.

% !TEX root = ../manyparticle.tex

\section{Case studies and numerical experiments}\label{app:case}
In this section, we verify the assumptions for the examples treated in Section~\ref{sec:case}.

\subsection{Loss functions}
We first give sufficient conditions to satisfy the assumptions on the loss $R$, when the Hilbert space is $\F = L^2(\rho)$ for a probability measure $\rho$ on a space $\X$, which is either a domain of $\R^d$ or the torus. In this setting, typical losses are the form $R(f) = \int r(x,f(x))\d \rho(x)$ for a function $r:\X\times\R\to\R_+$. The next lemma gathers some properties of such losses.
\begin{lemma}[Properties of the loss]\label{lem:loss}
If $r$ is convex in the second variable, then $R$ is convex. If $r$ is differentiable in the second variable with $\partial_2 r$ Lipschitz, uniformly in the first variable, then $R$ is differentiable with differential $dR$ Lipschitz. If moreover $\vert \partial_2 r\vert^2 \leq C_1 r + C_2$ for some constants $C_1,C_2>0$, then $dR$ is bounded on sublevel sets.
\end{lemma}
\begin{proof}
The convexity property is easy. If $\partial_2 r$ is $L$-Lipschitz (uniformly in the first variable), then it can be seen that $dR_f: h \mapsto \int r'(r,f(x))h(x)\d\rho(x)$ is the differential of $R$ because for all $f,h\in \F$, by a Taylor expansion,
\[
\vert R(f+h)-R(f)-dR_f(h)\vert \leq  \frac L2  \int \vert h(x)\vert^2 \d \rho(x) = \frac L2 \Vert h\Vert^2 = o(\Vert h\Vert).
\]
It is direct to see that $dR$ is $L$-Lipschitz in the operator norm. Finally, if $\vert \partial_2 r\vert^2 \leq C_1r+C_2$, then 
\[
\Vert dR_f\Vert^2 = \int \vert \partial_2 r(x,f(x))\vert^2\d \rho(x) \leq C_1 R(f) + C_2
\]
so $dR$ is bounded on sublevel sets.
\end{proof}

\subsection{Sparse deconvolution}

Let us show that Assumptions~\ref{ass:bounded} hold for the setting of Section~\ref{subsec:spikes}. While we did not mention explicitly the choice of the torus $ \Theta=\R^d/\mathbb{Z}^d$ as a domain in our results, it poses no difficulties: it is similar to the case $\Theta$ bounded, but without the difficulties related to boundaries. On the separable Hilbert space $\F = L^2(\Theta)$ where $\Theta$ is the $d$-torus endowed with the normalized Lebesgue measure, the loss $R$ is as in Lemma~\ref{lem:loss} with $r(x,f)= (f(y)-y(x))^2$ and the regularization term $\tilde V = 1$ trivially satisfies the assumptions. Let us turn our attention to the function $\phi(\theta): x \mapsto \psi(x-\theta)$. Its norm does not depend on $\theta$, so it is bounded. If $\psi$ is continuously differentiable with Lipschitz derivative, then $\phi$ is differentiable with $d\phi_\theta (\bar \theta): x \mapsto  \nabla \psi(x-\theta)\cdot \bar \theta$ which is bounded (again, its norm $\Vert d\Phi \Vert = \Vert \nabla \psi \Vert$ does not depend on $\theta$) and is Lipschitz, as similarly as in the proof of Lemma~\ref{lem:loss}.

It remains to check the Morse-type regularity assumption i.e., to check that for all $f\in \F$, the function $\theta \mapsto \langle f, \phi(\theta)\rangle = \int f(x) \psi(x-\theta) dx$ has a set of regular values which is dense in its range. If this function is constantly $0$ then this is trivially true, otherwise, its range is an interval of $\R$. By Morse-Sard's lemma, if this function is $d-1$-times continuously differentiable, then the set of critical values has zero Lebesgue measure and our assumption holds. By differentiating under the integral sign, this assumption is thus satisfied if $\varphi$ is $d-1$-times continuously differentiable.

\subsection{Neural network: sigmoid activation}\label{appsubsec:sigmoid}
Let us show that Assumptions~\ref{ass:bounded} hold for the setting presented in Section~\ref{subsec:neuralnet} in the case of sigmoid activation functions.
We write the disintegration of $\rho$ with respect to the variable $x$ as $\rho(\d x \otimes \d y) = \rho(\d y\vert x )\otimes \rho_x(\d x)$ where $\rho_x$ is the marginal of $\rho$ on $\X$ and $(\rho(\cdot \vert x))_{x \in \X}$ a family of conditional probabilities on $\R$ (see~\cite[Thm. 5.3.1]{ambrosio2008gradient}). On the separable Hilbert space $L^2(\rho_x)$, the loss $R$ is as in Lemma~\ref{lem:loss} with $r(x,p)= \int_\R \ell(p,y)\rho(\d y\vert x)$ and the regularization term $\tilde V = 1$ satisfies trivially the assumptions. In order to simplify notations, we consider the augmented variable $z = (x,1)\in \R^{d-1}$ and $\rho_z$ its distribution when $x$ is distributed according to $\rho_x$. Let $\phi(\theta): x \mapsto \sigma (z\cdot \theta)$, defined on $\Theta=\R^{d-1}$.
\begin{lemma}\label{lem:differentiablesigmoid}
If $\rho_x$ has finite moments up to order $4$, then the function $\phi: \R^{d-1}\to \F$ is differentiable with a Lipschitz and bounded differential $d\phi_\theta(h): x \mapsto (h \cdot z)\sigma'(z \cdot \theta)$ where $z=(x,1)$.
\end{lemma}
\begin{proof}
Let us check that the function $d\phi$ defined above is indeed the differential of $\phi$. For $\theta,h \in \R^{d-1}$, we have
\begin{align*}
\Delta(h)^2 &:= \Vert \phi(\theta+h)-\phi(\theta)-d\phi_\theta(h)\Vert^2\\
&= \int_\X \vert \sigma(\theta\cdot z + h\cdot z) - \sigma(\theta \cdot z)- (h\cdot z)\sigma'(z\cdot \theta)\vert^2 \d \rho_z(z)\\
&\leq \frac{L^2}{4}\int_\X \vert h\cdot z\vert^4 \d \rho_z(z)
\end{align*}
where $L$ denotes the Lipschitz constant of $\sigma'$.
So if $\rho_z$ has finite $4$-th order moment $M_4(\rho_z)$ then $\Delta(h) \leq \frac{L\sqrt{M_4(\rho_z)}}{2}\vert h\vert^2$ and $d\phi$ is indeed the differential of $\phi$. This differential is bounded and Lipschitz since $\Vert d\phi_\theta\Vert\leq \Vert \sigma'\Vert_\infty \sqrt{M_2(\rho_z)}$ and $\Vert d\phi_\theta - d\phi_{\tilde \theta}\Vert \leq L\sqrt{M_4(\rho_z)}\vert \theta-\tilde \theta\vert$ for all $\theta,\tilde \theta \in \R^{d-1}$. Finally, it is clear that if $\rho_x$ has finite $4$-th moment then so does $\rho_z$.
\end{proof}

It remains to check the Sard-type regularity assumption i.e., to check that for all $f\in \F$, $\theta \mapsto \langle f, \phi(\theta)\rangle = \int_\X f(x) \sigma((x,1)\cdot \theta)\d \rho_x (x)$ has a set of regular values which is dense in its range. If this function is constantly $1$ then this is trivially true, otherwise, its range is an interval of $\R$. If $\rho_x$ has finite moments up to order $2d-2$ then the function above is $d-1$ continuously differentiable and the conclusion follows by Morse-Sard's lemma. 

In the statement of Proposition~\ref{prop:sigmoid}, the boundary assumption is explicitly mentioned so the proof is complete. We now briefly explain why is it difficult to check the Sard-type regularity in the boundary condition a priori. Consider the simple setting of a quadratic loss $R(f) = \frac12 \Vert f-f^*\Vert_\F^2$ where $f^*$ is the optimal Bayes regressor that we may assume smooth. As required in the boundary assumption, consider a function $f\in \F$ of the form $f = R'(\int \Phi \d \mu) = \int \Phi \d \mu - f^*$ for some $\mu$ in the domain of the functional $F$. In the limit $r \to \infty$, the function $g_f(r\theta) \coloneqq \langle f, \phi(r\theta)\rangle = \int f(x) \sigma(r\theta \cdot (x,1))d \rho_x(x)$ converges to the function $\bar g_f(\theta) = \int_{\theta \cdot (x,1) \geq 0} f(x) d \rho_x(x)$. This function is continuously differentiable on the sphere if the density of $\rho_x$ is in $C_0(\R^{d-2})$ and $f$ is bounded and continuous (this is the case here) and the convergence of $g_f(r\cdot)\to \bar g_f$ is indeed in $C^1$. However, we cannot guarantee a very high regularity for $f$ in general: differentiating under the integral sign $d-1$-times requires to have moments of order $(d-1)$ bounded for $\mu$, which cannot be assumed a priori ($\mu$ is just known to be in the domain of $F$). This prevents us from applying Morse-Sard's lemma.

\subsection{Neural network: ReLU activation}\label{app:ReLU}
\subsubsection{Classical parameterization}
We now consider the activation function $\sigma(s) = \max \{ 0,s\}$ and let $\Phi(w,\theta): x\in \R^{d-2} \mapsto w \sigma((x,1) \cdot \theta)$ be defined on $\R\times\R^{d-1}$. We show in the next lemma that $\Phi$ is not differentiable on the whole space: at points where the $\theta$ coordinate vanishes, there is a discontinuity in the derivative. The consequence of this Lemma is that particle gradient flows (Definition~\ref{def:classicalGF})---and a fortiori Wasserstein gradient flows---are not well-defined in this case. 
\begin{lemma}\label{lem:nondifferentiablerelu}
If $\rho_x$ has finite moments up to order $2$ and has a density, then the function $\Phi: \R^{d}\to \F$ is differentiable on the set $\{ (w,\theta)\in \R\times \R^{d-1}\;;\; \theta \neq 0\}$, with differential $d\Phi_{(w,\theta)}(\bar w,\bar \theta): x \mapsto (\bar w + w \bar \theta \cdot z) \sigma'(z \cdot \theta)$ where $z=(x,1)$ and $\sigma'$ is the Heaviside step function. Yet, the differential $d\Phi$ is discontinuous at points of the form $(w,0)$ for $w\neq 0$.
\end{lemma}
\begin{proof}
Let us verify that the properties of a Fr\'echet differential are satisfied by the function $d\Phi$ above. For $u=(w,\theta)$ such that $\theta \neq 0$ and $\bar u = (\bar w, \bar \theta)$ in $\R^d$, we have
\begin{align*}
\Delta^2_{u}(\bar u) &:= \Vert \Phi (u+\bar u)-\Phi(u)-d\Phi_{u}(\bar u)\Vert^2 \\
&= \int \vert f(u+\bar u,x)-f(u,x)- d f_{(u,x)}(\bar u,0)\vert^2\d \rho_x(x)
\end{align*}
where we have introduced the function $f:(u,x)\mapsto w\sigma(\theta\cdot(x,1))$ which is differentiable whenever $\theta\cdot (x,1)\neq 0$. For $\theta \in \R^{d-1}\setminus\{0\}$ and $\epsilon>0$, we introduce the sets $S_{\theta,\epsilon} = \{x \in \R^{d-2}\;;\; \vert \theta \cdot (x,1) \vert\leq \epsilon \vert (x,1)\vert\}$ and decompose the previous integral in two parts: one where $f$ is regular and the integrand can be controlled with second order terms, and another one that deals with the non-differentiability inside $S_{\theta,\epsilon}$. This choice of definition for $S_{\theta,\epsilon}$ guarantees that we have $(\theta + \bar \theta)\cdot (x,1) \neq 0$ whenever $x$ is not in $S_{\theta,\bar \theta}$. This leads to
\begin{align*}
\Delta^2_{(w,\theta)}(\bar w,\bar \theta) &\leq \int_{S_{\theta,\vert \bar \theta\vert}}  4\vert (x,1) \vert^2\cdot \vert \bar u \vert^2 \cdot (2\vert u \vert + \vert \bar u \vert))^2\d \rho_x(x)
+ \int_{\R^{d-2}\setminus S_{\theta,\vert \bar \theta\vert}} \vert \bar w \bar \theta \cdot (x,1)\vert^2 \d \rho_x(x)\\
& \leq 4 \vert \bar u \vert^2 \cdot (2\vert u \vert + \vert \bar u \vert))^2 \int_{S_{\theta,\vert \bar \theta\vert}}\vert (x,1)\vert^2\d\rho_x(x) +  \frac12 \vert \bar u \vert^4 \int_{\R^{d-2}\setminus S_{\theta,\vert \bar \theta\vert}} \vert (x,1)\vert^2\d \rho_x(x)
\end{align*}
If $\rho_x$ has finite  second order moment $M_2(\rho_x)$, then the second term is negligible in front of $\vert \bar u\vert^2$ when $\vert \bar u\vert$ goes to $0$. In order to have the same property for the first term, it is sufficient that the integral $\int_{S_{\theta,\vert \bar \theta\vert}}\vert (x,1)\vert^2\d\rho_x(x)$ goes to $0$ as $\vert \bar \theta\vert$ goes to $0$ which is the case since $\rho_x$ has a density. Therefore, under these conditions, $\d \Phi_{(w,\theta)}$ is the differential of $\Phi$ at $(w,\theta)$.  
To exhibit a discontinuity, let $w\neq 0$ and $\theta \in \S^{d-1}$. For $t>0$, it holds
\[
\Vert d\Phi_{(w,t\theta)} - d\Phi_{(w,-t\theta)}\Vert^2 \geq  \vert w\vert^2 \int \vert \theta \cdot (x,1)\vert^2 \d \rho_x(x).
\]
For suitable choices of $\theta$ (for instance, $\theta=(0_{\R^{d-1}},1)$), this lower bound is strictly positive and independent of $t$.
\end{proof}

Although we do not use this fact explicitly in the paper, it is interesting to note that the regularizing potential $V:(w,\theta) \mapsto \vert w\vert\cdot \vert \theta\vert$ is admissible in the $2$-homogeneous setting of Assumptions~\ref{ass:2homogeneous}: although it is not differentiable nor convex, it is positively $2$-homogeneous and semiconvex.
\begin{lemma}
The function $V:(w,\theta)\mapsto \vert w\vert \cdot \vert \theta\vert$ defined on $\R \times \R^{d-1}$ is positively $2$-homogeneous and semi-convex.
\end{lemma}
\begin{proof}
The homogeneity property is clear, and to see that $V$ is semi-convex, it is sufficient to remark that 
\[
(w,\theta)\mapsto V(w,\theta) + \frac12 \vert w \vert^2 + \frac12 \vert \theta\vert^2 = \frac12 (\vert \theta\vert + \vert w \vert)^2
\]
is convex, since it is the square of a norm.
\end{proof}

\subsubsection{A differentiable parameterization}
We now consider the alternative parameterization considered in Proposition~\ref{prop:relu}, defined as $\Phi(\theta) : x \mapsto \sigma (s(\theta)\cdot (x,1))$ where $\sigma(t)=\max\{t,0\}$ and $s$ is the signed square function $s(t)=t\vert t\vert = \sign(t)\cdot t^2$ that acts entry-wise. As $\Phi$ is clearly positively $2$-homogeneous so we just have to prove the differentiability of $\Phi$, which is done with the same technique as in Lemma~\ref{lem:nondifferentiablerelu}.
%  $\Phi(\theta):x \mapsto w \sigma((x,1)\cdot s(\theta))$ where $s(\theta)_i = \theta_i \vert \theta_i\vert$ is the ``signed square'' function.

\begin{lemma}\label{lem:differentiablerelu}
If $\rho_x$ has finite moments up to order $2$ and has a density, then the function $\Phi: \R^{d}\to \F$ is differentiable, with differential $d\Phi_{\theta}(\bar \theta): x \mapsto 2(\sum_{i=1}^d \bar \theta_i \vert \theta_i\vert z_i) \sigma'(s(\theta)\cdot (x,1))$ where $\sigma'$ is the Heaviside step function.
\end{lemma}
\begin{proof}
As in Lemma~\ref{lem:nondifferentiablerelu}, we verify that the properties of a Fr\'echet differential are satisfied by the function $d\Phi$ above. First, $\Phi$ is differentiable at $0$ with differential $0$ since it is positively $2$-homogeneous. For $\theta \neq 0$ and $\bar \theta$ in $\R^d$, we have
\begin{align*}
\Delta^2_{\theta}(\bar \theta) &:= \Vert \Phi (\theta+\bar \theta)-\Phi(\theta)-d\Phi_{\theta}(\bar \theta)\Vert^2 \\
&= \int \vert f(\theta + \bar \theta,x)-f(\theta,x)- d f_{(\theta,x)}(\bar \theta,0)\vert^2\d \rho_x(x)
\end{align*}
where we have introduced the function $f:(\theta,x)\mapsto \sigma(s(\theta)\cdot(x,1))$ which is differentiable whenever $s(\theta) \cdot (x,1)\neq 0$. For $\theta \in \R^{d}\setminus\{0\}$ and $\epsilon>0$, we introduce the sets $S_{\theta,\epsilon} = \{x \in \R^{d-1}\;;\; \vert s(\theta) \cdot \vert (x,1)\vert  \leq \epsilon \vert (x,1)\vert \}$ and decompose the previous integral in two parts: one where $f$ is regular and the integrand can be controlled with second order terms (through Taylor-Lagrange inequality), and another one that deals with the non-differentiability inside $S_{\theta,\vert \theta\vert}$ (where $f$ is only Lipschitz, locally in $\theta$ and globally in $(x,1)$). This leads to the bounds, for some constants $C_\theta, C'_\theta >0$ and $\vert \bar \theta \vert$ small enough
\begin{align*}
\Delta^2_{\theta}(\bar \theta) &\leq C_\theta \vert \bar \theta\vert^2 \int_{S_{\theta,\vert \bar \theta\vert}} \vert (x,1)\vert^2  \d \rho_x(x)
+ C'_\theta \vert \bar \theta \vert^4 \int_{\R^{d-1}\setminus S_{\theta,\vert \bar \theta\vert}} \vert (x,1)\vert^2 \d \rho_x(x)
%& \leq 4 \vert \bar u \vert^2 \cdot (2\vert u \vert + \vert \bar u \vert))^2 \int_{S_{\theta,\vert \bar \theta\vert}}\vert (x,1)\vert^2\d\rho_x(x) +  \frac12 \vert \bar u \vert^4 \int \vert (x,1)\vert^2\d \rho_x(x)
\end{align*}
Under the assumption that $\rho_x$ has a density, we have that $\Delta^2_{\theta}(\bar \theta) = o(\vert \bar \theta\vert^2)$. Therefore, $\d \Phi_{\theta}$ is the differential of $\Phi$ at $\theta$.
\end{proof}
Note that the condition on the moments of $\rho_x$ is less strong for ReLU activation than for sigmoids in Lemma~\ref{lem:differentiablesigmoid}: this comes from the fact that ReLU is piece-wise linear. Similarly as what explained in the end of Section~\ref{appsubsec:sigmoid}, it is difficult to verify the Sard-type regularity assumption so it is left as an assumption in Proposition~\ref{prop:relu}.

% !TEX root = ../manyparticle.tex

\subsection{Numerical experiments : details and additional results}\label{app:details}

Animated plots of the particle gradient flows shown in this article may be found online at \url{https://lchizat.github.io/PGF.html}\footnote{These videos appear at this place in the official supplementary material of the NIPS 2018 publication, but had to be removed from the present version due to software incompatibility.}

\paragraph{Setting for the empirical particle-complexity plot.}
Here we give more details on the numerical experiments behind Figure~\ref{fig:particlecomplexity}. 
\begin{enumerate}
\item
For the leftmost panel, the setting is similar to that of Figure~\ref{fig:spikes}: for each realization, $5$ spikes are randomly distributed on the $1$-torus (with a minimum separation of $0.1$) with random weights between $0.5$ and $1.5$ and a small noise is added to the filtered signal. Then for each choice of $m$, we initialize $m$ particles on a regular grid on $\{0\} \times \Theta$ and integrate the particle gradient flow with the forward-backward algorithm until the improvement per iteration is below a small tolerance threshold.
\item
For the center panel, the setting is similar to that of Figure~\ref{fig:ReLU}, but here in dimension $d=100$. The data is normally distributed and the ground truth labels are generated by a similar neural network with $20$ neurons (with random normally distributed parameters). The objective function is the square loss without regularization, so the global minimum corresponds to a $0$ loss. We optimize using SGD with fresh samples at each iteration.
\item
The rightmost panel shows, similarly, the particle-complexity for training a neural network with a single hidden layer and sigmoid activation function, in dimension $d=100$. The data is distributed on a sphere and the ground truth labels are generated by a similar neural network with $20$ neurons with random normal weights.  Again, we minimize with SGD the square loss without regularization and the global minimum corresponds to a $0$ loss. 
\end{enumerate}

We compare the performance with the method of simply minimizing on the weights with the same initialization. This is a convex problem, and the minimum value attained does not depend on the minimization method. 
We plot for each case the final excess loss as a function of $m$ for several random realizations of the experiment and, for each value of $m$, its geometric average over all realizations.
We have indicated in transparent green the area of loss values which should be interpreted as ``optimal'' but are not exactly $0$ because the optimization has been stopped in finite time and the loss is not known exactly but estimated through sampling.

\paragraph{Choice of the initial weights in the partially $1$-homogeneous case.}
In all previous numerical experiments dealing with the partially $1$-homogeneous case, we have initialized the particle gradient flow on a discretization of $\{0\}\times \Theta$. But Theorem~\ref{th:mainbounded} allows for a large variety of initialization patterns. In this paragraph, we comment on the various possibilities and explain how the proof of Theorem~\ref{th:mainbounded} helps understanding why the corresponding particle-complexity is impacted.

We display on Figure~\ref{fig:offset} a sparse spikes deconvolution experiment, in a similar setting than in Figure~\ref{fig:spikes}, but with different initializations. For this problem, where $m_0=5$ spikes are to be recovered, we have observed numerically that the particle gradient flows initialized on a uniform grid on $\{0\}\times \Theta$ succeed in finding a global minimizer as soon as there are more than $m= 7$ particles. In the first panel of Figure~\ref{fig:offset}, the particle gradient flow with $m=15$ particles initialized on $\{1\}\times \Theta$ fails at finding a minimizer and a larger number of particles is needed for success (as shown in the center panel, with $m=30$). 

This phenomenon can be understood in light of the proof of Theorem~\ref{th:mainbounded}: when $(\mu_t)_t$ enters the neighborhood (given by Proposition~\ref{prop:escape1}) of the local minimum $\nu$ reached in the left panel, say at $t_0>0$, there exists a set $\R^-\times K^-$ such that if a particle of $\mu_{t}$ for $t>t_0$ falls in this set, then $(\mu_t)_t$ eventually escapes from this local minimum $\nu$. This set is, to put it simply, a $0$-sublevel set of the function $F'(\nu)$, which is a positively $1$-homogeneous function in the weight coordinate (the vertical axis in Figure~\ref{fig:offset}). The difficulty here is that, because of the initialization, $\mu_{t_0}$ is concentrated on $\R_+\times \Theta$, so we can only hope that a particle ``slides'' on a ridge of $F'(\nu)$ to eventually reach the set $\R^-\times K^-$. This is guaranteed to happen in the many-particle limit (this is the object of Lemma~\ref{lem:rideridge}), but this is likely to require a high density of particles around ridges of $F'(\nu)$ (the set $\R\times \{\theta_0\}$ in the proof of Theorem~\ref{th:mainbounded}). This supports the idea that initializing on $\{0\}\times \Theta$ is a good choice. In the rightmost panel of Figure~\ref{fig:offset} we also show the behavior for a uniform initialization on $(\{1\}\times \Theta) \cup (\{-1\}\times \Theta)$ which, in this example, also avoids the difficulty described above. 

\begin{figure}[h]
\centering
\begin{subfigure}{0.32\textwidth}
\includegraphics[scale=0.44,trim=0.2cm 0.2cm 0.2cm 0.2cm,clip]{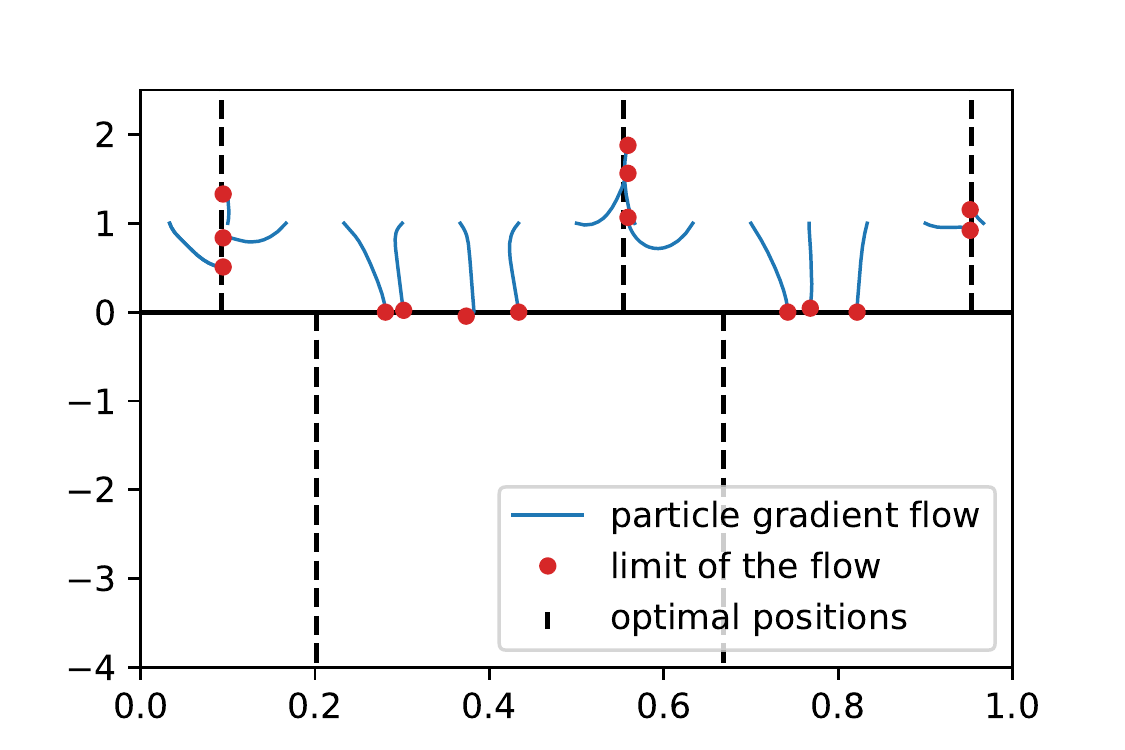}
\end{subfigure}
\begin{subfigure}{0.32\textwidth}
\includegraphics[scale=0.44,trim=0.2cm 0.2cm 0.2cm 0.2cm,clip]{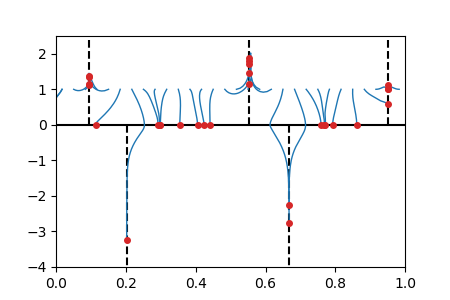}
\end{subfigure}
\begin{subfigure}{0.32\textwidth}
\includegraphics[scale=0.44,trim=0.2cm 0.2cm 0.2cm 0.2cm,clip]{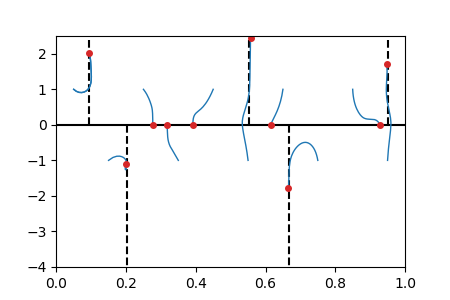}
\end{subfigure}
\caption{Particle gradient flow for partially $1$-homogeneous problems (sparse spikes recovery): effect of the initialization pattern on the particle-complexity. (left) $m=15$ particles on  $\{1\}\times \Theta$: failure (center) $m=30$ particles on $\{1\}\times \Theta$: success (right) $m=10$ particles on $(\{1\}\times \Theta) \cup (\{-1\} \times \Theta)$: success.}
\label{fig:offset}
\end{figure}

\end{document}